\documentclass[11pt,a4paper]{article}
\usepackage{amsfonts,amsmath,amsthm,amsxtra,amscd,amssymb}
\usepackage{amsmath}
\usepackage{amsthm}
\usepackage{amsxtra}
\usepackage{tikz}

\usepackage{amssymb,verbatim}
\usepackage{bm}
\usepackage{newlfont}

\usepackage{color}
\usepackage[T1]{fontenc}
\usepackage{mathptmx}
\usepackage{bm}
\usepackage{palatino}

\theoremstyle{plain}
\newtheorem{thm}{Theorem}[section]
\newtheorem{prop}[thm]{Proposition}
\newtheorem{lem}[thm]{Lemma}

\theoremstyle{definition}
\newtheorem{exmp}{Example}[section]
\newtheorem{defn}{Definition}[section]

\newtheorem{rem}{Remark}[section]

\def\bsy{\boldsymbol}

\def\ch #1{\mathrm{Char}\,#1}

\def\CV{{\mathcal{V}}}
\def\CH{{\mathcal{H}}}

\def\CE{{\mathcal{E}}}

\def\CI{{\mathcal{I}}}
\def\CC{{\mathcal{C}}}
\def\CZ{{\mathcal{Z}}}
\def\CU{{\mathcal{U}}}

\def\a{\alpha}
\def\b{\beta}
\def\D{\Delta}
\def\e{\varepsilon}
\def\g{\gamma}
\def\G{\Gamma}

\def\l{\lambda}
\def\L{\Lambda}
\def\k{\kappa}
\def\O{\Omega}

\def\o{\omega}
\def\O{\Omega}
\def\r{\rho}
\def\s{\sigma}

\def\t{\tau}
\def\th{\theta}

\def\B#1{\mathbb{#1}}

\def\P #1{\partial_{#1}}
\def\ds{\displaystyle}

\def\B#1{\mathbb{#1}}

\def\P #1{\partial_{#1}}
\def\ds{\displaystyle}

\begin{document}

\title{Symmetry Reduction, Contact Geometry and\\ Nonlinear Trajectory Planning}
\author{J. De Don\'{a}\footnote{School of Electrical Engineering \& Computer Science, University of Newcastle, N.S.W., Australia, 2308 },\ \    N. Tehseen\footnote{School of Mathematical and Physical Sciences, University of Newcastle, N.S.W., Australia, 2308},\  \ P.J. Vassiliou\footnote{Program in Mathematics and Statistics, Faculty of ESTeM, University of Canberra,  A.C.T., Australia, 2601}}
\maketitle
\begin{abstract}{\small  We study control systems invariant under a Lie group with application to the problem of nonlinear trajectory planning. A theory of symmetry reduction of exterior differential systems 
\cite{AF} is employed to demonstrate how symmetry reduction and reconstruction is effective in the explicit, exact construction of planned system trajectories. We show that, while a given control system  with symmetry may not be static feedback linearizable or even flat, it may nevertheless possess a flat or even linearizable symmetry reduction and from this, trajectory planning in the original system may often be carried out or greatly simplified. We employ the contact geometry of Brunovsky normal forms \cite{Vassiliou2006a}, \cite{Vassiliou2006b} to develop tools for detecting and analysing these phenomena. The effectiveness of this approach is illustrated by its application to a problem in the guidance of marine vessels. A method is given for the exact and explicit planning of surface trajectories of models for the control of under-actuated ships. It is shown that a 3-degrees-of-freedom control system for an under-actuated ship has a symmetry reduction which permits us to give an elegant explicit, exact solution to this problem.}\end{abstract}
\vskip 30 pt

\baselineskip=13 pt

\section{Introduction}

In this paper we study the problem of {\it trajectory planning} in those nonlinear control systems $\bsy{\o}$ invariant under a Lie transformation group, $G$. In this we emphasise the explicit, exact construction of planned system trajectories which can be obtained by exploiting the geometry of the quotient system $\bsy{\o}/G$ together with the symmetries of the parent control system $\bsy{\o}$. The rationale here is that the geometry of $\bsy{\o}/G$ is normally quite different from that of $\bsy{\o}$; in particular $\bsy{\o}/G$ may be linearizable or flat, even when $\bsy{\o}$ is not.

In relation to the class of all smooth abstract control systems
\begin{equation}\label{abstractControlSystem}
\bsy{\dot{x}}=\bsy{f}(t,\bsy{x},\bsy{u})
\end{equation}   
the subset of those admitting Lie symmetries is certainly a very thin one. By contrast however, among control systems of interest in real applications a great many do admit such symmetries. This is because in applications such control systems arise from physical or geometrical considerations which come with inherent symmetries such as, for instance, Galilean invariance in the case of many mechanical systems. In fact, it is sometimes surprising to discover just how rich the Lie symmetry group of a seemingly intractable control system can be. We will see examples of this phenomenon herein. These considerations motivate the study of control systems with Lie symmetry.

As is well known,  Lie symmetry {\it reduction} has played a significant role in constructing explicit solutions of differential equations over a very long period, \cite{Olver}. However, its role in control theory specifically is of more recent date.  Various applications of Lie symmetry theory to control systems  intensified around the early 1980s in works such as \cite{NijmeijerVanDerSchaft}, \cite{GrizzleMarcus}, \cite{GardnerShadwick90}, and continued in the more recent \cite{BlochKMarsden},  \cite{Respondek02}, \cite{Ohsawa13} and references therein. Of these,  {\it symmetry reduction} of control systems is specifically discussed in  the foundational papers \cite{GrizzleMarcus},\ \cite{BlochKMarsden} and in  the more recent works \cite{Ohsawa13} where a goal is the role of symmetry reduction in optimal control problems. A textbook account of symmetry reduction in nonlinear control is given in \cite{Elkin}. 

The aim of this paper is to explore the role of symmetry reduction and reconstruction, as formulated by Anderson and Fels \cite{AF}, in addressing the problem of trajectory planning in nonlinear control systems. By trajectory planning we mean the solution to the following problem. Fix a smooth control system  (\ref{abstractControlSystem}) and a connected submanifold $S\subseteq X$ where $X$ is the manifold of all states. Prescribe any smooth, immersed  curve $t\mapsto \g(t)\subset S$  parametrized by time $t$. Determine the time-parametrized curve in the control space $t\mapsto \bsy{u}(t)$ that realises $\g$. 
That is, a time-parametrized curve $t\mapsto \bsy{u}(t)$ in control space is the solution of the trajectory planning problem if $\bsy{x}(t)=\g(t)$ is the unique solution of 
$\bsy{\dot{x}}=\bsy{f}(t,\bsy{x},\bsy{u}(t))$. Thus the trajectory planning problem has three pieces of initial data: the smooth control system (\ref{abstractControlSystem}), a connected submanifold $S\subseteq X$ and a smooth, time-parametrized curve in $S$.    The choice of submanifold $S$ will normally be determined by the actual control problem that requires solution. There can be a vast difference in the explicit solvability of the problem depending on the choice of $S$. In this paper we are focussed not only in establishing the existence of a suitable solution but on creating tools for the explicit exact construction of the solution. 

A class of control systems for which the problem of trajectory planning can often have a particularly elegant, straightforward solution is that of static feedback linearizable systems. These are systems 
(\ref{abstractControlSystem}) for which there exists a static feedback transformation\footnote{Strickly speaking, transformation (\ref{staticFeedbackTransformation}) of 
non-autonomous system (\ref{abstractControlSystem}) is what we later refer to as an {\it extended static feedback transformation}, as in the literature such transformations are usually restricted to time-invariant systems.}
\begin{equation}\label{staticFeedbackTransformation}
(t,\ \bsy{x},\ \bsy{u})\mapsto \Big(t,\ \phi(t,\bsy{x}),  \ \psi(t,\bsy{x},\bsy{u})\Big)
\end{equation}  
which identifies (\ref{abstractControlSystem}) with some Brunovsky normal form ({\it see} section \ref{contactFlat}). The trajectories of any given Brunovsky normal form are straightforward to explicitly and completely describe and transformation 
(\ref{staticFeedbackTransformation}) can then often be used to plan trajectories for (\ref{abstractControlSystem}). There is a wider class of control systems which contain the static feeback linearizable ones as a special subset and having the same useful properties in relation to trajectory planning, namely the so called {\it flat systems}, \cite{LevineBook}.  If a control system (\ref{abstractControlSystem}) is not static feedback linearizable but there is a prolongation of it (obtained by successive differentiation of one or more of the controls) which is static feedback linearizable then the system is likewise flat. The problem of bounding the  number of prolongations in a control system, in order that it be rendered static feedback linearizable, is largely solved in the work of Sluis and Tilbury \cite{SluisTilbury} and Chetverikov \cite{Chetverikov07}. These authors provide a bound on the ``number of integrators`` that must be added to (\ref{abstractControlSystem}) so that the prolonged system is static feedback linearizable. For a large class of control systems these results provide an algorithmic test for linearization via prolongation and therefore the class of control systems for which trajectory planning is often within reach. Therefore, a question of interest in this paper is to explore the problem of trajectory planning for control systems which are {\it not flat}  (therefore not linearizable after sufficient prolongation) or else when such linearization of the system does not prove to be helpful for planning trajectories in prescribed submanifolds $S\subset X$ of state space $X$. 

\subsection{Summary of the main contributions of the paper} This paper establishes a number of new results and concepts which have
both theoretical and practical importance. Firstly, the reduction theory of Anderson and Fels \cite{AF} has not previously been applied to control systems. It has the advantage over other theories in that the reconstruction of the dynamics of the original control system from that of its symmetry reduction (quotient) has been given a clear and comprehensive geometric formulation.  Moreover, the differential system that mediates the reconstruction of the trajectories from those of the quotient system is shown to be  a {\it system of Lie type}, \cite{Bryant95}, \cite{Doubrov}. Lie type systems have special properties such as superposition formulas and a separate reduction procedure. An important property is that if the symmetry group of the control system is solvable then the reconstruction of the system dynamics from those of its quotient system can be reduced to quadrature. Secondly, assuming one knows the Lie algebra of infinitesimal symmetries that generate the action, we show how the contact geometry of Brunovsky normal forms can be applied and developed to give a computationally effective algorithm for quickly detecting when a given invariant control system has a static feedback linearizable quotient as well as deducing its basic properties. Importantly, this can be achieved {\it without constructing} the quotient which, in general, is not an algorithmic process. That is, constructing the quotient explicitly requires integration even if the infinitesimal symmetries are known.\footnote{It is shown in \cite{AF} that many properties of the quotient can be inferred from the algebraically constructible semi-basic forms. However,  to construct the quotient one is required to know a cross-section which, in turn,  requires us to know the action. In general determining the action in a given problem usually requires integration.  } Additionally, we establish the notion of {\it control admissible symmetries} which identifies the class of Lie symmetries such that the quotient of control system (\ref{abstractControlSystem}) is itself a control system thereby generalizing the notion of state space symmetries.  The approach developed 
in this paper is very explicit, capable of creating trajectory planning methodologies for general invariant control systems -- both flat and non-flat; variational and non-variational. 
Finally, the developed theory is illustrated on an important class of control systems, namely marine guidance systems and we derive explicit, novel planning methodologies for surface trajectories. 

\subsection{Outline} Section 2 mainly gives a brief summary of the theory of exterior differential systems with symmetry, from \cite{AF}, that we shall require. Section 3 begins by describing the geometric formulation \cite{Vassiliou2006a}, \cite{Vassiliou2006b} of a central object in geometric control theory, namely Brunovsky normal forms which are mathematically identical to the contact systems on the partial prolongations of jet space $J^1(\B R, \B R^q)$. It ends by deriving necessary and sufficient conditions for when the class of all equivalences between a control system and its linearization contains a static feedback transformation. Section 4 establishes necessary and sufficient conditions for the existence of {\it linearizable quotients} of a control system from knowledge only of the Lie algebra of its infinitesimal symmetries. It is pointed out that this begins to address a long-standing problem: given a control system, characterise its ``maximal linearizable subsystems`` and use these to obtain control theoretic information about the parent system. The notion of {\it control admissible symmetries} is introduced in this section and the main consequence is established. The theoretical developments of sections 3 and 4 are illustrated by providing a detailed example of the symmetry reduction of a non-flat control system in 2 controls arising from a result of \cite{MartinRouchon94}.

Section 5 is devoted to an application of the foregoing results to the problem of trajectory planning in control systems for under-actuated marine vessels. We study one such control system in complete detail and show that it is flat (contrary to a claim in the literature) but that flatness, of itself,  does not solve the most immediate problem of planning surface trajectories. It is shown however that symmetry reduction and reconstruction does permit the required trajectory planning.


\section{Exterior Differential Systems with Symmetry}
In this section we briefly describe the results that we shall require from Anderson and Fels, \cite{AF}. As pointed out by these authors, when carrying out symmetry reduction of Pfaffian systems one ordinarily needs to extend one's study beyond the category of Pfaffian systems themselves to incorporate exterior differential systems (EDS) more broadly. This is because it is possible for the symmetry reduction of a Pfaffian system to be non-Pfaffian, as pointed out in \cite{AF}. To keep the exposition within bounds as well as incorporating interesting applications, we will only consider the Pfaffian case in this paper.

An EDS on smooth manifold $M$ or {\it exterior differential system} is a graded ideal in the ring of all differential forms on $M$, denoted $\O^*(M)$. Let $\O^k(M)$ denote the set of all differential $k$-forms on $M$. 
The EDS $\mathcal{I}$ consists of a direct sum of subsets $\mathcal{I}^k\subset\O^k(M)$
$$
\mathcal{I}=\mathcal{I}^1\oplus\mathcal{I}^2\oplus\cdots\oplus\mathcal{I}^n
$$
where $n=\dim\,M$. Not all the subsets $\mathcal{I}^j$ need be non-empty. It is the structure of $\mathcal{I}$ as direct sum that prompts the qualification {\it graded}. We do not have elements of $\mathcal{I}$ that have, for instance, the form  $\psi^{\ell_1}+\psi^{\ell_2}$ where $\psi^{\ell_1}$ is a $\ell_1$-form and $\psi^{\ell_2}$ is a $\ell_2$-form, with $\ell_1\neq \ell_2$. Because we shall only consider Pfaffian systems in this paper, our ideals consist of degree 1 and degree two components $\CI=\CI^1\oplus\CI^2$ in which the elements of $\CI^2$ consist of the exterior derivatives of the forms in $\CI^1$. We will therefore not usually mention $\CI^2$ and refer to the Pfaffian system by its degree 1 component which will often by denoted by $\bsy{\o}$.

\begin{defn}
Let $\mu : G\times M\to M$ be an action of a Lie group $G$ on a smooth manifold $M$. Let $\bsy{\o}$ be a Pfaffian system on $M$. 
\begin{enumerate}
\item Pfaffian system $\bsy{\o}$ is said to be {\it $G$-invariant} if $\mu_g^*\bsy{\o}\subseteq\bsy{\o}$ for all $g\in G$, where
$\mu_g(x)=\mu(g,x)$, for $x\in M,\ g\in G$.  One also says that $G$ is the {\it symmetry group} of $\CI$. 
\item If $\CV$ is a smooth vector field distribution on $M$ then we say that it is {\it $G$-invariant} or that $G$ constitutes the {\it symmetries of} $\CV$ if ${\mu_g}_*\CV\subseteq \CV$ for all $g\in G$.
\item The $G$-action is said to be {\it free} if whenever $x\in M$ satisfies $\mu_g(x)=x$ then $g=e:=\text{id}_G$
\end{enumerate}
\end{defn}

\begin{defn}
A solution or {\it integral submanifold} of $\bsy{\o}$ on $M$ is a mapping $s: \mathbb{R}^p\to M$ such that $s^*\th=0$ for all $\th\in\bsy{\o}$. The domain of $s$ being $p$-dimensional implies that the dimension of the image of $s$ is (at most) $p$. 
\end{defn}
It follows from these definitions that if $s$ is an integral manifold of $\bsy{\o}$ then so is $\mu_g\circ s$. All actions of Lie group $G$ on smooth manifold $M$ will be assumed to be {\it regular}, (\cite{Olver}, p23; pp213--218), such that the quotient of $M$ by the action of $G$ is a smooth manifold denoted $M/G$ together with a smooth surjection $\mathbf{q}:M\to M/G$ which assigns each point of $M$ to its $G$-orbit. 
We will always assume that $M/G$ has the Hausdorff separation property. From now on $G$ will always denote a Lie group acting smoothly and regularly on smooth manifold $M$. 
\vskip 5 pt


To motivate the next definition, suppose that $\mathbf{q}:M\to M/G$ is the quotient map for the action of Lie group $G$ on $M$. It induces the pullback $\mathbf{q}^*: \O^*(M/G)\to\O^*(M)$ from forms on $M/G$ to those on $M$. Hence if $\bsy{\o}$ is a Pfaffian system on $M$, we can ask about the existence of a Pfaffian system $\bsy{\bar{\o}}$ on $M/G$ such that $\mathbf{q}^*\bsy{\bar{\o}}\subseteq\bsy{\o}$. We shall see that such a Pfaffian system on the quotient of $M$ by $G$ holds significance for applications to control theory.

\begin{defn}\label{QuotientDef}
Let $\mathbf{q}:M\to M/G$ be the quotient map for the smooth, regular Lie group action on $M$ by $G$. Let $\bsy{\o}$ be a Pfaffian system on $M$ that is invariant under the action of $G$. That is, $G$ is a symmetry group for $\bsy{\o}$. The {\it quotient of $\bsy{\o}$ by the given $G$-action} is the maximal Pfaffian system $\bsy{\bar{\o}}$ on $M/G$ such that for all $\bar{\th}\in\bsy{\bar{\o}}$,\ \ 
$\mathbf{q}^*\bar{\th}\in\bsy{\o}$.
\end{defn}

It is instructive to explore this definition further. The following definition is central. 

\begin{defn}[\cite{AF}]\label{DefCrossSection}
Let $M$ be a smooth manifold with a smooth, regular action of Lie group $G$. Let $\mathbf{q}:M\to M/G$ be the quotient map.
\begin{enumerate}
\item A map $\s:M/G\to M$ is said to be a {\it cross-section} (to the action of $G$) if
$\mathbf{q}\circ\s:M/G\to M/G$ is the identity map on $M/G$.
\item Let $\CV\subset TM$ be a distribution and let $\bsy{\G}$ be the Lie algebra of infinitesimal generators of the action of $G$. We say that $G$ is {\it transverse} to $\CV$ if $\bsy{\G}\cap\CV=\{0\}$. Also,
if $\CV=\ker\bsy{\o}$, we say that $G$ is {\it transverse} to $\bsy{\o}$.
\item Let $\CI$ be an exterior differential system on $M$. The {\it semi-basic k-forms} $\CI^k_\mathrm{sb}$ satisfy $\th(\bsy{\G})=0$, for all $\th\in\CI^k$.
\end{enumerate}
\end{defn}
 
As proven in \cite{AF}, many properties of the quotient $\bsy{\bar{\o}}$ can be deduced from the semi-basic forms. In particular, if $\s$ is a cross-section for the action of $G$ then the quotient 
$\bsy{\bar{\o}}:=\bsy{\o}/G$ is equal to $\s^*\bsy{\o}_{\text{sb}}$, where $\bsy{\o}_\mathrm{sb}$, denotes the semi-basic 1-forms.  In this paper we are not concerned with exterior differential systems in general but only with {\it Pfaffian systems}.  
As pointed out previously, the quotient of Pfaffian system need not be a Pfaffian system. As we wish to stay exclusively in the category of Pfaffian systems we note the following characterisation, proven in 
\cite{AF} (Lemma 7.2) that if $\bsy{\o}$ is a Pfaffian system then its $G$-quotient, $\bsy{\o}/G$, is a Pfaffian system if and only if its EDS of semi-basic forms is also a Pfaffian system. 

The following provides further motivation for Definition \ref{QuotientDef}.

\begin{lem}\label{semiBasicDistribution}
Let $\o$ be a Pfaffian system on $M$ invariant under the smooth, regular action of a Lie group $G$ and $\mathbf{q}:M\to M/G$ the quotient map. Let distribution $\CV=\ker\bsy{\o}$. Then $\mathbf{q}_*\CV=\ker\bsy{\bar{\o}}$. 
\end{lem}

\noindent{\it Proof.}  Suppose $\bar{X}\in\mathbf{q}_*\CV$. Then there exist $X\in\CV$ with $\bar{X}=\mathbf{q}_*X$ and we have, for all $\bar{\th}\in\bsy{\bar{\o}}$, 
$\bar{\th}(\bar{X})=\bar{\th}(\mathbf{q}_*(X))=\left(\mathbf{q}^*\bar{\th}\right)(X)=\th(X)=0$ for some $\th\in\bsy{\o}$ by Definition \ref{QuotientDef}, so  $\mathbf{q}_*\CV\subseteq \ker\bsy{\bar{\o}}$.

Suppose $\bar{Y}\in\ker\bsy{\bar{\o}}$. Then $\bar{\th}(\bar{Y})=0$ for all $\bar{\th}\in\bsy{\bar{\o}}$. With cross-section $\s$ we have $\bar{\th}=\s^*\th_{\text{sb}}$, some $\th_{\text{sb}}\in\bsy{\o}_{\text{sb}}$ and hence  $0=\left(\s^*\th_{\text{sb}}\right)(\bar{Y})=\th_{\text{sb}}(\s_*\bar{Y})$. Now it is possible to show (see proof of Theorem \ref{quotientCheck}) that $\ker\bsy{\o}_{\text{sb}}=\CV\oplus\bsy{\G}$ and hence
$\s_*\bar{Y}\in\CV\oplus\bsy{\G}$ whence $\bar{Y}=\mathbf{q}_*\s_*\bar{Y}\in\mathbf{q}_*(\CV\oplus\bsy{\G})=\mathbf{q}_*\CV$. Hence $\ker\bsy{\bar{\o}}\subseteq\mathbf{q}_*\CV$.
\hfill\qed

\vskip 5 pt
While symmetry reduction can produce a control system on a lower dimensional manifold, this doesn't necessarily imply that the quotient system is more tractable to trajectory planning than the parent system; in fact the opposite is often the case. Therefore for effective use of symmetry reduction in control we focus on the situation in which $\bsy{\o}$ may not be flat or static feedback linearizable and then seek a 
quotient $\bsy{\bar{\o}}$ that is at least flat, if not static feedback linearizable.  Subsequently, we use this quotient to solve the trajectory planning problem for $\bsy{\o}$. Similar aims are mentioned in \cite{Murray97} though the interesting approach proposed there is different from the one taken in this paper. Here more attention is paid to the problem of constructing the trajectories of $\bsy{\o}$ from those of $\bsy{\bar{\o}}$. Particular attention is paid to algorithmically characterising desirable properties of $\bsy{\bar{\o}}$, such as linearizability, in terms of the infinitesimal symmetries of $\bsy{\o}$.

\subsection{Reduction and reconstruction}

It is shown in \cite{AF} that provided the action of the symmetry group of $\bsy{\o}$ satisfies certain constraints and certain objects can be explicitly constructed then trajectories of 
$\bsy{\o}$ can be constructed from those of $\bsy{\bar{\o}}$ by solving an ODE of {\it Lie type}. Such ODE systems form a particularly nice class since they themselves have a reduction theory which aids their solution \cite{Bryant95} , \cite{Doubrov}. 

\begin{thm}[\cite{AF}]\label{mainThm}
Let $\bsy{\o}$ be a Pfaffian system on manifold $M$  that is invariant under the smooth, regular action of a Lie group $G$ and such that the $G$-action is free and transverse to $\bsy{\o}$. Let $\bsy{\bar{\o}}$ be the quotient of 
$\bsy{\o}$ on  $M/G$ by the $G$-action and let $\bar{s}:\CU\to M/G$ be an integral manifold of $\bsy{\bar{\o}}$.  Then
\begin{enumerate}
\item There exists a cross-section to the action $\s: M/G\to M$ to the $G$-action.
\item There exists a map $g: \CU\to G$  such that $s(x)=\mu(g(x),(\s\circ\bar{s})(x))$ is an integral submanifold of $\bsy{\o}$ for all $x\in \CU$, where $\mu:G\times M\to M$ is the $G$-action.
\item The function $g$ is constructed by solving a Frobenius integrable system of Lie-type.
\end{enumerate}  

\end{thm}

Provided we can explicitly satisfy the hypotheses of Theorem \ref{mainThm} and solve the resulting system of Lie type then every solution of the quotient gives rise to a solution of the original Pfaffian 
system $\bsy{\o}$ on $M$; all the trajectories of $\bsy{\o}$ can be constructed in this way. Note that Theorem \ref{mainThm} is a specialisation of the corresponding theorem in \cite{AF}, adapted to our situation.   Another important specialisation we make is that as we are dealing with control systems we will always require our group actions to {\it preserve the independent variable}, usually {\it time} $t$; 
see section \ref{LinearQuotients}.


\section{Contact Geometry and Static Feedback Linearization}\label{contactFlat}

Since Brunovsky's seminal paper \cite{Brunovsky} there has been a great deal of interest in the exact linearization of nonlinear control systems. Early works were those of Krener \cite{Krener73}; Hunt, Su, Meyer \cite{HuntSuMeyer}; Brocket \cite{Brocket78} and Jakubzcyk and Respondek \cite{Respondek80} which provided successful approaches to the problem of finding a transformation of a controllable nonlinear system to its linear equivalent, if one existed. 

Precisely, the problem of linearization is to determine when a nonlinear control system  (\ref{abstractControlSystem}) can be transformed by a static feedback transformation 
(\ref{staticFeedbackTransformation}) to some member of the family of {\it Brunovsky normal forms} which can be viewed as the family of Pfaffian systems ${B}(\k_1,\ldots,\k_m)$,
\begin{equation}\label{BNF}
\begin{aligned}
&\omega^1_0=dx^1_0-x^1_1dt,\ \ \omega^1_1=dx^1_1-x^1_2dt,\ \ldots\ \ \ ,\omega^1_{\k_1}=dx^1_{\k_1}-v_1dt,\cr
&\omega^2_0=dx^2_0-x^2_1dt,\ \ \omega^2_1=dx^2_1-x^2_2dt,\ \ldots\ \ \ ,\omega^2_{\k_2}=dx^2_{\k_2}-v_2dt,\cr
&\hskip 60 pt \cdots\cdots\hskip 60 pt\cdots\cdots\cr
&\hskip 60 pt \cdots\cdots\hskip 60 pt\cdots\cdots\cr
&\omega^m_0=dx^m_0-x^m_1dt,\ \ \omega^m_1=dx^m_1-x^m_2dt,\ \ldots,\ \ \ \omega^m_{\k_m}=dx^m_{\k_m}-v_mdt,
\end{aligned}
\end{equation}
each labelled by the sequence of positive integers $\k_1, \k_2, \ldots, \k_m$. Relabelling controls $v_j$ as $x^j_{\k_j+1}$ one notices immediately that the Brunovsky normal forms are identical to the {\it partial prolongations} of the contact system 
$$
\{\o^j_0=dx^j_0-x^j_1\,dt\}_{j=1}^m
$$ 
on the jet space $J^1(\B R, \B R^m)$ in which $\o^j_0$ is prolonged to order $\k_j+1$. 
Dually in standard coordinates we have the {\it contact distribution} on $J^1(\B R, \B R^m)$,
\begin{equation}
\mathcal{C}\langle m\rangle=\left\{\P x+z^1_1\P {z^1}+z^2_1\P {z^2}+\cdots+z^m_1\P {z^n},\ \P {z^1_1},\ \P {z^2_1},\ \ldots,\P {z^m_1}\right\},
\end{equation}
occasionally denoted by $\mathcal{C}^1_m$. Partially prolonging to arbitrary order in any or all ``directions`` $\P {z^j}$ produces a controllable linear control system in Brunovsky normal form.   
We shall adopt the notation $\mathcal{C}\langle\r_1,\r_2,\ldots,\r_k\rangle$ for the Brunovsky normal form in which $\r_\ell$ denotes the ``number of variables of order $\ell$``.  Note that a {\it partial prolongation} of the contact distribution on $J^1(\B R,\B R^q)$, the {\it contact distribution} on a partial prolongation of jet space $J^1(\B R,\B R^q)$ and a {\it Brunovsky norm form} are different ways of refering to the same  object. These terms  are used interchangably in this paper.

\subsection{Geometry of partial prolongations }\label{secGGNF}
While the problem of static feedback linearization is considered by control theorists to be solved, it appears that the field has by and large not emphasised its intimate relation to contact geometry. A number of important practical and conceptual benefits accrue when the underlying contact geometry of linearizable nonlinear control systems is given greater prominence as we point out below. A key classical theorem that animates the direction of a good deal of the work in linearization of control systems is the Goursat normal form and its generalizations. 
In due course, it raised a question in differential geometry as to how much of the local structure of a sub-bundle  $\mathcal{V}\subset TM$ of the tangent bundle to smooth manifold $M$ can be encoded by {\it numerical invariants} associated to its {\it derived flag}; see \cite{Bryant79}, \cite{Yamaguchi}. Such a set of numerical invariants is referred to as the {\it derived type} of $\mathcal{V}$. It is rarely the case that the local structure of $\CV$ is determined by its derived type. However the situations for which the derived type {\it is} a complete local invariant include the basic theorems of differential geometry such as the Frobenius theorem, the Pfaff theorem, the Engel normal form and the Goursat normal form. The aforementioned works \cite{Bryant79}, \cite{Yamaguchi} have added several important new cases to this short list, effectively providing a geometric characterisation of the contact systems on jet spaces $J^k(\B R^n,\B R^m)$.  The simplest example of a differential system for which the derived type {\it does not} determine its local structure is the family of generic 2-plane fields on any 5-manifold, \cite{Cartan10}. In this section we discuss one case in which the derived type is a complete local invariant, the  generalised Goursat normal form \cite{Vassiliou2006a}, \cite{Vassiliou2006b}. This theorem provides a geometric characterisation of the {\it partial prolongations} of the jet space $J^1(\B R, \B R^q)$ and {\it inter alia} a geometric formulation of linearization in nonlinear control systems.  That is to say a geometric formulation of Brunovsky normal forms. It is pointed out that the problem of static feedback linearization of control systems (regardless of their local form) can be viewed as a refinement of this general research program in relation to derived type. This reformulation unifies and extends a number of the standard results in control theory such as the GS algorithm \cite{GardnerShadwick92}, the extended Goursat normal form \cite{BushnellTilburySastry93}, the Hunt-Su-Meyer \cite{HuntSuMeyer} and Respondek-Jakubczyk \cite{Respondek80} linearization theorems and related results. Of immediate relevance for this paper are the benefits it confers on the question of finding and analysing {\it linearizable symmetry reductions} of control systems. 

Let us recall the classical Goursat normal form. Let $\CV\subset TM$ be a smooth, rank 2 sub-bundle over smooth manifold $M$ such that $\CV$ is {\it bracket generating} and $\dim\CV^{(i)}=2+i$ while 
$\CV^{(i)}\neq TM$.  Then there is a generic subset $\widehat{M}\subseteq M$ such that in a neigbourhood of each point of $\widehat{M}$ there are local coordinates $x,z_0,z_1,z_2,\ldots z_k$ such that $\CV$ has local expression
\begin{equation}\label{CtR2R}
\CC^{(k)}_1=\Big\{\P x+\sum_{j=1}^k z_j\P {z_{j-1}},\ \P {z_k}\Big\} 
\end{equation} 
where $k=\dim M-2$. That is, $\CV$ is locally equivalent to $\mathcal{C}^{(k)}_1$ on $\widehat{M}$.

The Goursat normal form  solves the recognition problem of when differential system $\CV$ can be identified with the contact distribution (\ref{CtR2R}) via a local diffeomorphism of $M$ in terms of a property of its derived flag. The generalised Goursat normal form (GGNF) does the same
job in case distribution (\ref{CtR2R}) is replaced by the {\it partial prolongations} of the contact distribution on jet space $J^1(\mathbb{R},\mathbb{R}^q)$, with $q>1$,
\begin{equation}\label{CtR2Rq}
\mathcal{C}^{(1)}_q=\Big\{\P x +z^1_1\P {z^1}+z^2_1\P {z^2}+\ldots+z^q_1\P {z^q},\P {z^1_1},\ldots,\P {z^q_1}\Big\}.
\end{equation}
This is a much more delicate task involving more subtle invariants but the end result is an analogous theorem.  An example of a partial prolongation of (\ref{CtR2Rq}) is given by 
\begin{equation}\label{Ctppr011}
\mathcal{C}\langle 0,1,1\rangle=\Big\{\P x+z^2_1\P {z^2}+z^2_2\P {z^2_1}+z^3_1\P {z^3}+z^3_2\P {z^3_1}+z^3_3\P {z^3_2}, \P {z^2_2}, \P {z^3_3}\Big\}
\end{equation}
in which there is one ``dependent variable of order 2`` and one of order 3 (so $q=2$). The notation  $\mathcal{C}\langle 0,1,1\rangle$ denotes one dependent variable of order 2 (second element) and one dependent variable of order 3 (third element) and zero dependent variables of order 1 (first element). Note that in (\ref{Ctppr011}) the superscript 2 or 3 denotes the order of the variable. 
\vskip 10 pt
We are now ready to describe the aforementioned theorem on partial prolongations. This leads to a procedure based on the {\it refined} derived type of a sub-bundle for recognising sub-bundles as partial prolongations and for constructing equivalences to them that puts no restriction on the local form of the sub-bundle under consideration.  We begin with an introduction to the basic tools required.

Suppose $M$ is a smooth manifold and $\CV\subset TM$ a smooth sub-bundle of its tangent bundle. The structure tensor is the map $\delta :\L^2\CV\to TM/\CV$
defined by
$$
\delta(X,Y)=[X,Y]\mod\CV,\ \text{for all}\ \ X,Y\in\CV.
$$
In more detail, suppose
$X_1,\ldots,X_r$ is a basis for $\CV$ and $\o^1,\ldots,\o^r$ is the dual basis for its dual $\CV^*$. Suppose $Z_1,\ldots,Z_s$ is a basis for $TM/\CV$ such that 
$[X_i,X_j]\equiv c^k_{ij}Z_k\mod \CV$ for some functions $c^k_{ij}$ on $M$. Then $\delta=c^k_{ij}\o^i\wedge\o^j\otimes Z_k$; that is, a section of $\L^2\CV^*\otimes TM/\CV$.  
The structure tensor encodes important information about a sub-bundle, the most obvious of which is the extent to which it fails to be Frobenius integrable.

Let us define the map
$
\zeta:\CV\to\text{Hom}(\CV,TM/\CV)\ \ \text{by}\ \ \zeta(X)(Y)=\delta(X,Y)
$
 For each $x\in M$, let
$
\mathcal{S}_x=\{v\in\CV_x\backslash 0~|~\zeta(v)\ \text{has less than generic rank}\}.
$
Then $\mathcal{S}_x$ is the zero set of homogeneous polynomials and so defines a subvariety of the projectivisation 
$\mathbb{P}\CV_x$ of $\CV_x$. We shall denote by Sing$(\CV)$ the fibre bundle over $M$ with fibre over $x\in M$ equal to 
$\mathcal{S}_x$ and we refer to it as the {\it singular variety} of $\CV$. For $X\in \CV$ the matrix of the homomorphism $\zeta(X)$ will be called the {\it polar matrix} of $[X]\in\mathbb{P}\CV$. There is a map $\text{deg}_\CV:\mathbb{P}\CV\to
\mathbb{N}$ well defined by
$
\text{deg}_\CV([X])=\text{rank}~\zeta(X)\ \ \ \text{for}\ \ [X]\in\mathbb{P}\CV.
$
We shall call $\text{deg}_\CV([X])$ the {\it degree} of $[X]$ (relative to $\CV$). Function $\text{deg}_\CV([X])$ is a diffeomorphism invariant: $\text{deg}_{\phi_*\CV}([\phi_*X])=\text{deg}_\CV([X])$. Hence the singular variety $\text{Sing}{(\CV)}$ is also a diffeomorphism invariant.

The computation of the singular variety for any given sub-bundle  $\CV\subset TM$ is algorithmic. It involves only differentiation and commutative algebra operations. One computes the {\it determinantal variety} of the polar matrix for generic $[X]$; see \cite{Vassiliou2006a}, \cite{Vassiliou2006b}, \cite{Vassiliou2009a} for more details and examples. 

Recall that an important invariant object associated to any distribution is its Cauchy bundle. For $\CV\subseteq TM$ by
$\ch\CV$ we denote the Cauchy characteristics of $\CV$; that is,
$
\ch\CV=\{X\in\CV~|~[X,\CV]\subseteq\CV\}.
$
If $\CV$ is such that all derived bundles $\CV^{(j)}$ and all their Cauchy characteristics
$\ch\CV^{(j)}$ have constant rank on $M$ then we say that $\CV$ is {\it totally regular}. In this case we refer to $\ch\CV^{(j)}$ as the {\it Cauchy bundle} of $\CV^{(j)}$. 

The aforemention singular bundle has a natual counterpart in which the quotient $\CV/\ch\CV$ replaces $\CV$. In this case, if $\text{Sing}{(\CV/\ch\CV)}$ is not empty then each of its points has positive degree.


\vskip 10 pt
\begin{defn}[Resolvent bundle]

  Suppose $\CV\subset TM$ is totally regular of rank $c+q+1$, $q\geq 2, c\geq 0$,
$\dim M=c+2q+1$. Suppose further that $\CV$ satisfies
\vskip 2 pt
\begin{itemize}
\item[\rm a)] $\dim\ch\CV=c$, $\CV^{(1)}=TM$
\item[\rm b)] 
$\text{Sing}(\widehat{\CV})_{|_x}=\mathbb{P}\widehat{\mathcal{B}}_{|_x}\approx\mathbb{R}
\mathbb{P}^{q-1}$, for each $x\in M$ and some rank $q$ sub-bundle $\widehat{\mathcal{B}}\subset\widehat{\CV}$.  
Then we call  $(\CV,\mathbb{P}\widehat{\mathcal{B}})$  a
{\it   Weber structure} of rank $q$ on $M$. 
\end{itemize} 
Given a  Weber structure
$(\CV,\mathbb{P}\widehat{\mathcal{B}})$, let ${R}(\CV)\subset\CV$, denote the largest sub-bundle such that
$$
{\pi}\big(R(\CV) \big)=\widehat{\mathcal{B}}.
$$
We call the rank $q+c$ bundle $R(\CV)$ the {\it resolvent bundle} associated to the  Weber structure $({\CV},\mathbb{P}\widehat{\mathcal{B}})$. The bundle $\widehat{\mathcal{B}}$ determined by the singular variety of 
$\widehat{\CV}$ will be called the {\it singular sub-bundle} of the  Weber structure. A  Weber structure
will be said to be {\it integrable} if its resolvent bundle is integrable.
\end{defn}
An {\it integrable}  Weber structure descends to the quotient of $M$ by the leaves of $\ch\CV$ to be the contact bundle on $J^1(\mathbb{R},\mathbb{R}^q)$. 

It is important to relate a given partial prolongation to its derived type. For this it is convenient to introduce the notions of {\it velocity} and {\it deceleration} of a sub-bundle.

\begin{defn}\label{decel}
Let $\CV\subset TM$ be a totally regular sub-bundle.
The {\it velocity} of $\CV$ is the ordered list of $k$ integers 
$$
\text{\rm vel}(\CV)=\langle\D_1,\D_2,\ldots,\D_k\rangle,\ \ \text{where}\ \ \D_j=m_j-m_{j-1},\ 1\leq j\leq k,
$$
where $m_j=\dim\CV^{(j)}$.

The {\it deceleration} of $\CV$ is the ordered list of $k$ integers 
$$
\text{\rm decel}(\CV)=\langle -\D^2_2,-\D^2_3,\ldots,-\D^2_k, \D_k\rangle,\ \ \text{where}\ \  \D^2_j=\D_j-\D_{j-1}.
$$
\end{defn}

The notions of velocity and deceleration are refinements of the well known {\it growth vector} of a sub-bundle. If we think of the growth vector as a type of displacement or distance then the notions of velocity and deceleration acquire a natural meaning. We will see that the deceleration vector is a complete invariant of a partial prolongation except when $\D_k>1$, in which case one must also add that the resolvent bundle be integrable.
  

To recognise when a given sub-bundle has
or has not the derived type of a partial prolongation we introduce one further canonically associated sub-bundle that plays a crucial role. 
\vskip 5 pt
\begin{defn} If $\CV\subset TM$ is a totally regular sub-bundle of derived length $k$ we let 
$\ch\CV^{(j)}_{j-1}$ denote the intersections
$$
\ch\CV^{(j)}_{j-1}=\CV^{(j-1)}\cap\ch\CV^{(j)},\ 1\leq j\leq k-1.
$$
Let
$$
\chi^j_{j-1}=\dim\ch\CV^{(j)}_{j-1},\ \ 1\leq j\leq k-1.
$$  
We shall call the integers $\{\chi^0,\chi^j,\chi^j_{j-1}\}_{j=1}^{k-1}$ the {\it type numbers} of $\CV\subset TM$ and the list
$$
\mathfrak{d}_r(\CV)=\big[\,\big[m_0,\chi^0\big],\big[m_1,\chi^1_0,\chi^1\big],\big[m_2,\chi^2_1,\chi^2\big],\ldots,\big[m_{k-1},\chi^{k-1}_{k-2},\chi^{k-1}\big],\big[m_k,\chi^k\big]\,\big]
$$ 
as the {\it refined derived type} of $\CV$.
\end{defn}
\vskip 5 pt
It is easy to see that in every partial prolongation sub-bundles $\ch\CV^{(j)}_{j-1}$ are non-trivial and integrable, an invariant property of $\CV$. Furthermore, there are simple relationships between the type numbers in any partial prolongation thereby providing further invariants for the local equivalence problem. 
\vskip 5 pt
\begin{prop} [\cite{Vassiliou2006a}]\label{derivedtype} Suppose $\CV$ is a partial prolongation of the contact distribution of $\mathcal{C}\langle q\rangle$ on $J^1(\B R, \B R^q)$. Then the type numbers 
$m_j,\ \chi^j,\ \chi^j_{j-1}$ comprising the refined derived type $\mathfrak{d}_r(\CV)$ satisfy
$$
\begin{aligned}
&\chi^j=2m_j-m_{j+1}-1,\ 0\leq j\leq k-1,\cr
&\chi^i_{i-1}=m_{i-1}-1,\ 1\leq i\leq k-1,
\end{aligned}
$$
where $k$ is the derived length of $\CV$.
\end{prop}

\begin{defn}
We say that  $\CV\subset TM$ has {\it refined derived type of a partial prolongation} (of the contact distribution on $J^1(\B R,\B R^q)$) if its type numbers $m_j, \chi^j, \chi^i_{i-1}$ are those of some partial prolongation, which then necessarily satisfy the equalities in Proposition \ref{derivedtype}. 
\end{defn}

\begin{defn} \label{GoursatBundle}
A totally regular sub-bundle $\CV\subset TM$ of derived length $k$ will be called a {\it Goursat bundle} with deceleration $\s$ if
\begin{enumerate}
\item $\CV$ has the derived type of a partial prolongation, with signature $\s=\text{\rm decel}(\CV)$ 
\item Each intersection $\ch\CV^{(i)}_{i-1}$ is an integrable sub-bundle. 
\item  In case $\D_k>1$, then $\CV^{(k-1)}$ determines an integrable Weber structure of rank $\D_k$ on $M$, where $k$ is the derived length of $\CV$.
\end{enumerate}
\end{defn}

The recognition problem for partial prolongations is solved by the generalised Goursat normal form.

\begin{thm}[Generalised Goursat Normal Form, \cite{Vassiliou2006a}]\label{GGNF} 
Let $\CV\subset TM$\ 
be a Goursat bundle over manifold $M$, of derived length $k>1$ and signature 
$\s=\mathrm{ decel}(\CV)$. 
Then there is an open, dense subset $\widehat M\subseteq M$ such that the restriction of $\CV$  to 
$\widehat M$ is locally equivalent to $\mathcal{C}(\s)$. Conversely any partial prolongation of 
$\mathcal{C}^{(1)}_q$ is a Goursat bundle.
\end{thm}

A partial prolongation is generically classified, up to a local diffeomorphism of the ambient manifold, by its deceleration vector. For this reason the deceleration of a Goursat bundle $\CV$ will sometimes be called its {\it signature}, a unique identifier of its local diffeomorphism class. If $\CV$ is a Goursat bundle and non-negative integer $\rho_j$ is the $j^\text{\rm th}$ component of its signature, then $\CV$ is locally diffeomorphic to a partial prolongation with $\rho_j$ ``dependent variables of order $j$". The theorem has a counterpart which provides an efficient procedure, {\tt\bf Contact}, for constructing an equivalence to $\mathcal{C}(\s)$ where $\s=\text{decel}(\CV)$ is the signature of $\CV$. Procedure {\tt\bf Contact} is described in detail in \cite{Vassiliou2006b} and will be used in this paper. The basic result is that one characterises the total differential operators and the function spaces that are used to generate the equivalences by differentiation. The first integrals of the resolvent bundle and those of the {\it fundamental bundles} $\Xi^{(j)}_{j-1}(\CV)/\Xi^{(j)}(\CV)$ 
\cite{Vassiliou2006b} are the required functions, where $\Xi^{(j)}(\CV)=\text{ann}\,\ch\CV^{(j)}$ and  $\Xi^{(j)}_{j-1}(\CV)=\text{ann}\,\ch\CV^{(j)}_{j-1}$.  If $\D_k=1$ then the resolvent is replaced by another integrable distribution, $\Pi^k$, whose construction is described in the example below.

\begin{exmp}
As a simple illustration, let us revisit the Hunt-Su-Meyer example \cite{HuntSuMeyer} using the generalised Goursat normal form.
The distribution in question is
$$
\CV=\Big\{\P t+\sin x_2\P {x_1}+\sin x_3\P {x_2}+(u_1+x_4^3)\P {x_3}+(x_4+x_4^3-x_1^{10})\P {x_4}+u_2\P {x_5},\ \P {u_1},\ \P {u_2} \Big\}
$$
We compute the refined derived type\footnote{Special purpose software based on the {\bf\tt Maple} package {\bf\tt DifferentialGeometry} \cite{AndersonTorre} computes this in about 3 seconds and simultaneously determines whether or not $\CV$ is a Goursat bundle. If it is, it outputs the signature of the equivalent Brunovsky normal form. } 
$$
\mathfrak{d}_r(\CV)=\big[\,\big[3,0\big], \big[5,2,2\big], \big[7,4,5\big], \big[8,8\big]\,\big]. 
$$
It is easily checked that $\mathfrak{d}_r(\CV)$ is the refined derived type of a contact distribution with signature $\text{decel}(\CV)=\langle 0, 1, 1\rangle$. That is, the type numbers $m_j, \chi^j$ and $\chi^j_{j-1}$ satisfy Proposition \ref{derivedtype} with signature $\langle 0,1,1\rangle$. Since, as we shall see $\ch\CV^{(2)}_1$ is integrable, we conclude that $\CV$ is locally diffeomorphic to of the Brunovsky normal form 
${B}\langle 2,3\rangle\simeq \mathcal{C}\langle 0,1,1\rangle$,
$$
\begin{aligned}
&dx_0-x_1dt,\ dx_1-v_1dt,\cr
&dy_0-y_1dt,\ dy_1-y_2dt,\ dy_2-v_2dt.
\end{aligned} 
$$
This settles the recognition problem for $\CV$.

We now use the proceedure {\tt\bf Contact} to construct an equivalence to this Brunovsky form. The derived length is $k=3$ and since $\D_3=1$, we must compute $\Pi^3$ as well as 
any nontrivial fundamental bundles; see procedure {\bf\tt Contact\,A}, step b) in \cite{Vassiliou2006b}.  From the refined derived type we see that the only non-trivial fundamental bundle is 
$$
\ds \Xi^{(2)}_1/\Xi^{(2)}=\{dx_4\}.
$$
and hence $x_4$ is a {\it fundamental function of order} 2. With $\Xi^{(2)}=\{dt, dx_1, dx_2\}$ we can choose $t$ for the independent variable in which case the operator of total differentiation can be taken to be 
$$
Z=\P t+\sin x_2\P {x_1}+\sin x_3\P {x_2}+(u_1+x_4^3)\P {x_3}+(x_4+x_4^3-x_1^{10})\P {x_4}+u_2\P {x_5}\in \CV.
$$
We can now compute $\Pi^3$ inductively by starting with $\Pi^1=\ch\CV^{(1)}_0=\ch\CV^{(1)}=\{\P {u_1}, \P {u_2}\}$. Then $\Pi^{i+1}=\Pi^i+[Z,\Pi^i], 1\leq i\leq k-1$. In this case we obtain
$$
\text{ann}\,\Pi^3=\{dt, dx_1\}
$$
and hence the {\it fundamental function of highest order} 3 is $x_1$. In the context of flat systems we would say that $x_1$ (the function to be differentiated 3 times) and $x_4$ (the function to be differentiated twice) are the flat outputs. Hence we obtain, with $z^{1,2}_0=x_4$, $z^{1,3}_0=x_1$,

$$
\begin{aligned}
&z^{1,3}_0=x_1,\ \ \ \ \ \ 
z^{1,3}_1=\mathbf{Z}z^{1,3}_0=\sin {x_2},\ \ \ \ \ 
z^{1,3}_2=\mathbf{Z}z^{1,3}_1=\cos x_2\sin x_3,\cr 
&z^{1,3}_3=\mathbf{Z}z^{1,3}_2=\cos x_2\cos x_3(u_1+x_4^3)-\sin^2x_3\sin x_2,\ \ \ \ \ 
z^{1,2}_0=x_4,\ \ \ \ z^{1,2}_1=\mathbf{Z}z^{1,2}_0=x_5+x_4^3-x_1^{10},\cr  
&z^{1,2}_2=\mathbf{Z}z^{1,2}_1=u_2+3x_4^2(x_5+x_4^3-x_1^{10})-10x_1^9\sin x_2.
\end{aligned}
$$
\end{exmp}

\subsection{Extended static feedback transformations} The generalised Goursat normal form  primarily solves the problem of {\it general equivalence} of a differential system $(M, \CV)$ to a partial prolongation $\mathcal{C}(\s)$ of the contact distribution $\mathcal{C}^1_q$. That is, the existence of {\it some} local diffeomorphism $\varphi:M\to J^\s(\B R,\B R^q)$ such that $\varphi_*\CV=\CC\langle\s\rangle$. If such an equivalence between $(M, \CV)$ and $(J^\s(\B R, \B R^q), \mathcal{C}(\s))$ exists then there also exists an infinite dimensional family of equivalences since the contact transformations form an infinite Lie pseudogroup. If $(M, \CV)$ is a control system then it is of great importance to know that within the infinite dimensional family of equivalences at least one can be chosen to be an extended static feedback transformation, which is a natural and simple generalisation of static feedback transformation to the case of {\it time-varying} control. 

\begin{defn}[Extended static feedback transformations]
A local diffeomorphism of the manifold of states, controls and time, $\bsy{x},\ \bsy{u},\ t$ of the form
$$
t\mapsto t,\ \bsy{x}\mapsto \bsy{\bar{x}}=\bsy{\a}(t, \bsy{x}),\ \bsy{u}\mapsto \bsy{\bar{u}}=\bsy{\b}(t, \bsy{x},\bsy{u})
$$
identifying a pair of control systems $\{\P t+\bsy{f}(t,\bsy{x},\bsy{u})\P {\bsy{x}},\ \P {\bsy{u}}\}$ and $\{\P t+\bsy{\bar{f}}(t,\bsy{\bar{x}},\bsy{\bar{u}})\P {\bsy{\bar{x}}},\ \P {\bsy{\bar{u}}}\}$
will be called an {\it extended static feedback transformation} (ESFT).   
\end{defn}

The existence of an ESFT identifying a control system to a Brunovsky normal form can be usefully established in terms of the generalised Goursat normal form.

\begin{thm}\label{StatFeedbackLin}
Let $\CV=\{\P t+\bsy{f}(t,\bsy{x},\bsy{u})\P {\bsy{x}},\ \P {\bsy{u}}\}$ be a control system defining a totally regular sub-bundle of $TM,\ \dim M=n+q+1$. Suppose $(M, \CV)$ is a Goursat bundle so that it is equivalent to a partial prolongation
$\mathcal{C}(\s)$ via local diffeomorphism $\varphi: M\to J^\s(\B R, \B R^q)$,  
$\varphi_*\CV=\mathcal{C}(\s)$, with derived length $k>1$. Then $\varphi$ can be chosen to be an extended static feedback transformation if and only if 
\begin{enumerate}
\item[1)] $\{\P {\bsy{u}}\}\subseteq\ch\CV^{(1)}_0$
\item[2)] $dt\in\mathrm{ ann}\,\ch\CV^{(k-1)}$  
\end{enumerate}
\end{thm}
\vskip 4 pt
\noindent{\it Proof.} Suppose $\varphi$ is an ESFT identifying $(M, \CV)$ with partial prolongation $\mathcal{C}( \s )$. Then $\varphi$ has the form
\begin{equation}\label{equivalence}
(t,\ \bsy{x},\ \bsy{u})\mapsto (t,\ \bsy{\a}(t, \bsy{x}),\ \bsy{\b}(t, \bsy{x},\bsy{u}))=(x,\ \bsy{z}^p_{j_p},\ \bsy{z}^p_{k_p}),\ \ 0\leq j_p\leq k_p-1,
\end{equation}
where the transformation $(t,\ \bsy{x})\mapsto (t,\ \bsy{\a}(t, \bsy{x}))$ is also a local diffeomorphism. Here $p$ is an index for the contact coordinates of order $k_p$; the largest of these is equal to the derived length of $\CV$. In contact coordinates $x$ can be taken to be the parameter along the trajectories of $\mathcal{C}(\s)$ and it is easy to see that it is a first integral of $\ch\mathcal{C}(\s)^{(k-1)}$. Hence
$\varphi^*x=t$ is a first integral of $\ch\CV^{(k-1)}$ since $\varphi$ identifies Cauchy bundles.  
It is straightforward to see that $\bsy{z}^p_{j_p}$, $0\leq j_p\leq k_p-1$ are first integrals of 
$\ch\mathcal{C}(\s)^{(1)}_0$ and hence $\varphi^*\bsy{z}^p_{j_p}=\bsy{\a}(t,\bsy{x})$ span the first integrals of $\ch\CV^{(1)}$. 
The elements of $\ch\CV^{(1)}_0$ are spanned by vector fields
$$
Y=T\P t+\sum_{i=1}^nA^i\P {x_i}+\sum_{\ell=1}^qB_\ell\P {u_\ell}
$$
We have $d\bsy{\a}(Y)=0$, and we deduce that $T=0$ (since $t$ is a first integral of $\ch\CV^{(1)}\subset\ch\CV^{(k-1)}$) and
$$ 
\frac{\partial (\a^1,\ \a^2,\ldots,\ \a^n)}{\partial (x_1,\ x_2,\ \ldots,\ x_n)}\bsy{A}=0
$$
where $\bsy{A}=\left(\begin{matrix}  A^1 & A^2 & \cdots & A^n\end{matrix}\right)^T$. It follows that $\bsy{A}=0$ since the components of $\bsy{\a}$ are functionally independent. Hence $\ch\CV^{(1)}_0$ contains vector fields of the form $Y=\sum_{\ell=1}^qB_\ell\P {u_\ell}$ only.
Let $H$ be the set of all vector fields of the form $\{Y_s=B_s^\ell\P {u_\ell}\}\subseteq\ch\CV^{(1)}_0$ which have the $n+1$ functions $(t,\ \bsy{\a}(t,\bsy{x}))$ as functionally independent first integrals. Because $\ch\CV^{(1)}_0$ is Frobenius, we have $H^{(\infty)}\subseteq\ch\CV^{(1)}_0$. Suppose
$H^{(\infty)}\neq \{\P {u_1},\, \P {u_2},\, \ldots\, \P {u_q} \}$. Then there is a first integral of $H$ which has $u$-dependence which contradicts the functional form of $\varphi$.

Conversely suppose hypotheses $1)$ and $2)$ of the theorem statement hold with $(M,\CV)$ a Goursat manifold so that local diffeomorphism $\varphi$ exists which identifies $\CV$ with partial prolongation 
$\mathcal{C}(\s)$. By procedure {\tt\bf Contact}, \cite{Vassiliou2006b}, pp 286--287, hypothesis $2)$ implies that $\varphi^*x=t$ can be taken to be an independent variable in the image system $\CC\langle\s\rangle$; that is, a parameter along the trajectories of $\CV$.   Further, according to procedure {\tt\bf Contact}, the components of the transformation to Brunovsky normal are constructed by differentiating the fundamental functions of order $j$, $\psi^{\ell_j,j}_0$, $1\leq \ell_j\leq \r_j$,  by the total differential operator $\mathbf{Z}$: 
$$
\psi^{\ell_j,j}_0,\ \  \psi^{\ell_j,j}_1=\mathbf{Z}\psi^{\ell_j,j}_0,\ \ \ldots,\ \ \psi^{\ell_j,j}_{j}=\mathbf{Z}\psi^{\ell_j,j}_{j-1}.
$$
The proof of Theorem 4.2 in \cite{Vassiliou2006b} shows that the functions $\{\psi^{\ell_j,j}_s\}_{s=0}^{j-1}$ are first integrals of $\ch\CV^{(1)}_0$. Hypothesis 1) in the theorem statement therefore allows us to conclude that $\varphi$ has the form (\ref{equivalence}).  Hence $\varphi$ is an ESFT. \hfill\qed



\section{Linearizable Quotients. Flat Quotients.}\label{LinearQuotients}
If an invariant control system is not static feedback linearizable or even flat, it is desirable to know of the existence of static feedback linearizable quotients. Otherwise, it is desirable to know of the existence of flat quotients. In either case it would be useful to have an {\it a priori} algorithmic test for the existence of such quotients. That is, without the necessity of constructing the quotient first. 

In this section we use the results of section \ref{contactFlat} to give such an {\it a priori} check for the existence of static feedback linearizable quotients only using as data knowledge of the Lie algebra of infinitesimal generators of the symmetries of the control system.  
To establish this, we introduce a slight refinement of the notion of Goursat bundle.

\begin{defn}\label{relativeGoursatBundle}
Let $(M, \CV)$ be totally regular distribution over manifold $M$. We say that $\CV$ is a {\it relative Goursat bundle} if its type numbers satisfy Proposition \ref{derivedtype} and items 2. and 3. of Definition 
\ref{GoursatBundle} are satisfied. 
\end{defn}

\begin{prop}\label{relativeGoursatProp1}
Let $(M, \CV)$ be a relative Goursat bundle with type numbers $m_j,\ \chi^j,\ \chi^i_{i-1}$ and suppose $\dim\ch\CV^{(1)}_0=c$. Then $\CV/\ch\CV^{(1)}_0$ is a Goursat bundle and the type numbers $\bar{m}_j,\ \bar{\chi}^j,\ \bar{\chi}^i_{i-1}$ of $\CV/\ch\CV^{(1)}_0$ satisfy 
\begin{equation}\label{CauchyCorrected}
m_j-\bar{m}_j=\chi^j-\bar{\chi}^j=\chi^i_{i-1}-\bar{\chi}^i_{i-1}=c.
\end{equation}
\end{prop} 
\vskip 0 pt
\noindent{\it Proof.} Suppose $\dim M=m$ and $\varphi^1(x),\ldots,\varphi^{m-c}(x)$ span the first integrals of $\ch\CV^{(1)}_0$. Extend these functions by $\psi^1(x),\ldots,\psi^c(x)$ to a local coordinate system on $M$. 
Denote by $\tau: M\to M$ the local diffeomorphism defined by 
$$
x\mapsto (\varphi^1(x),\ldots,\varphi^{m-c}(x),\ \psi^1(x),\ldots,\psi^c(x))=(y^1,\ldots,y^{m-c}, z^1,\ldots,z^c).
$$
Then we can write
\begin{equation}\label{quotientDecomp}
\tau_*\CV=\CH\oplus\t_*\ch\CV^{(1)}_0
\end{equation} 
in which $\CH=\t_*\CV\,\text{mod}\ \t_*\ch\CV^{(1)}_0$. As the first $m-c$ components of $\t$ are those of the quotient map $\mathbf{q}: M\to M/\ch\CV^{(1)}_0$, 
it is not hard to see that there is a basis change in $\CH$, 
$\{\bar{X}_1,\ldots,\bar{X}_n\}$ such that $[\P {z_\ell},\bar{X}_j]=0$, where $\t_*\ch\CV^{(1)}_0=\{\P {z_1},\ldots,\P {z_c}\}:=\CZ$. Hence $\CH=\mathbf{q}_*\CV$ is the quotient of $\CV$ by the leaves of the foliation induced by $\ch\CV^{(1)}_0$. We have expressed $(M,\CV)$ in a local trivialization
$(U\times Z, \CH\oplus\CZ)$, where $U\subset M/G$  and we have $[\CH^{(i)},\CZ^{(i)}]=0$ and $\CH^{(i)}\cap\CZ^{(i)}=0$ for all $i\geq 0$. We deduce that 
\begin{equation}\label{lemEq1}
\begin{aligned}
\dim\CV^{(i)}=&\dim\CH^{(i)}+c,\  \dim\ch\CV^{(j)}=\dim\ch\CH^{(j)}+c,\cr
  &\dim\ch\CV^{(i)}_{i-1}=\dim\ch\CH^{(i)}_{i-1}+c
\end{aligned}
\end{equation}
Distribution $\CH$ has type numbers $\bar{m}_j,\ \bar{\chi}^j,\ \bar{\chi}^i_{i-1}$ while  $\t_*\CV=\CH\oplus\CZ$ has type numbers
$m_j,\ \chi^j,\ \allowbreak \chi^i_{i-1}$. The relation between the two sets of type numbers follows from (\ref{lemEq1}) and the type numbers of $\CV/\ch\CV^{(1)}_0$ are those of a partial prolongation and hence, a Goursat bundle. \hfill\qed

\begin{prop}\label{relativeGoursatProp2}
If $(M, \CV)$ is a relative Goursat bundle then $\mathrm{decel}\,\CV=\mathrm{decel}\,\left(\CV/\ch\CV\right)$. 
\end{prop}

\noindent{\it Proof.} From the previous Proposition, we have $\bar{m}_j=m_j-c$. Since the deceleration of any bundle is a first or second difference of the derived flag bundle ranks, $\bar{m}_j, m_j$, we deduce that the decelerations of $\CV$ and $\CV/\ch\CV^{(1)}_0$ are identical.\hfill\qed


\begin{thm}[Existence of feedback linearizable quotients]\label{quotientCheck}
Let $\CV\subset TM$ be a subundle over smooth manifold $M$ that is invariant under the smooth, regular action of a Lie group $G$ with Lie algebra $\bsy{\G}$ of infinitesimal generators that are transverse to $\CV$ and suppose $\ch\CV^{(1)}_0=\{0\}$. Then 
\begin{enumerate}
\item[1)] The semi-basic 1-forms for the $G$-action satisfy $\ker\bsy{\o}_{\mathrm{sb}}=\CV\oplus\bsy{\G}$
\item[2)] If $(M, \CV\oplus\bsy{\G})$ is a relative Goursat bundle then the quotient $(M/G, \CV/G)$ of $(M, \CV)$ is locally equivalent to a partial prolongation of the contact distribution on  $J^1(\B R,  \B R^q)$. 
\item[3)] If $(M, \CV\oplus\bsy{\G})$ is a relative Goursat bundle then signature $\s=\mathrm{decel}\big(\CV\oplus\bsy{\Gamma}\big)=\mathrm{decel}(\CV/G)$ whence $(M/G, \CV/G)\simeq \mathcal{C}\langle \s\rangle$.
\end{enumerate}
\end{thm}
\vskip 4 pt
\noindent {\it Proof.} Let $\bsy{\o}=\text{ann}\,\CV$, $\mathbf{q}:M\to M/G$ the quotient map and $\bsy{\bar{\o}}=\bsy{\o}/G$ the quotient of $\bsy{\o}$ by the action of $G$. Recall that 
$
\mathbf{q}^*\bsy{\bar{\o}}\subset\bsy{\o};  
$
in particular $\CV\subset \ker\mathbf{q}^*\bsy{\bar{\o}}$.
We also have $\bsy{\G}\subset\ker\mathbf{q}^*\bsy{\bar{\o}}$ since $\left(\mathbf{q}^*\bsy{\bar{\o}}\right)(v)=\bsy{\bar{\o}}(\mathbf{q}_*v)=\bsy{\bar{\o}}(0)=0$ for all $v\in\bsy{\G}$.
Hence $\CV\oplus\bsy{\G}\subseteq\ker\mathbf{q}^*\bsy{\bar{\o}}$.  
\vskip 3 pt
\begin{lem}\label{littleLemma} 
$\mathbf{q}^*\bsy{\bar{\o}}=\bsy{\o}_\mathrm{sb}$.  
\end{lem}

\noindent{\it Proof of Lemma \ref{littleLemma}.} 
Since $(\mathbf{q}^*\bsy{\bar{\o}})(\bsy{\G})=0$ we have $\mathbf{q}^*\bsy{\bar{\o}}\subseteq \bsy{\o}_\mathrm{sb}$. We invoke the following elementary fact. If 
$f:M\to N$ is a smooth surjective submersion, with $\dim N=n$ and $\bsy{\Psi}\subset T^*N$ is a rank $k\leq n$ sub-bundle, then $f^*\bsy{\Psi}\subset T^*M$ has rank $k$. 
Now, from \cite{AF}, Theorem 5.1, we have that $\dim\bsy{\bar{\o}}=\dim\bsy{\o}-\dim\bsy{\G}=\dim\bsy{\o}_\mathrm{sb}$ and granting the elementary fact Lemma \ref{littleLemma} is proven. \hfill $\blacksquare$

\vskip 3 pt
Thus, 
$
\dim\ker\mathbf{q}^*\bsy{\bar{\o}}=\dim M-\dim\mathbf{q}^*\bsy{\bar{\o}}=\dim M-\dim\bsy{\o}_\mathrm{sb}=\dim\CV+\dim\bsy{\G}=\dim\left(\CV\oplus\bsy{\G}\right).
$
We have therefore proven that
$
\ker\mathbf{q}^*\bsy{\bar{\o}}=\ker\bsy{\o}_\mathrm{sb}=\CV\oplus\bsy{\G}, 
$
which is item $1)$. 
Now suppose that $\CV\oplus\bsy{\G}$ is a relative Goursat bundle. We have $\ch\left(\CV\oplus\bsy{\G}\right)=\bsy{\G}$ and hence by application of Proposition \ref{relativeGoursatProp1}, 
$$
\mathbf{q}_*(\CV\oplus\bsy{\G})=\mathbf{q}_*\CV
$$
has the refined derived type of a partial prolongation and hence is a relative Goursat bundle, proving item $2)$. By Theorem \ref{GGNF} and Proposition \ref{relativeGoursatProp2}, the quotient 
$\mathbf{q}_*\CV/\ch\left(\mathbf{q}_*\CV\right)$ is locally equivalent to $\CC\langle\s\rangle$ where $\s=\text{decel}(\CV\oplus\bsy{\G})$, proving item $3)$.\hfill\qed

\begin{rem}
Theorem \ref{quotientCheck} asserts that the existence of a static feedback linearizable quotient can be checked algorithmically from the refined derived type of 
$\CV\oplus\bsy{\G}$: the kernel of the semi-basic 1-forms. In particular, explicit construction of the quotient $\mathbf{q}_*\CV$ is unnecessary. Ordinarily integration is required if the action is not known or else only known infinitesimally. 
\end{rem}

\subsection{Control morphisms and linearizable quotients} We investigate the extent to which the quotient of a control system $(M, \CV)$ by its Lie symmetry group $G$ is also a control system on the quotient $M/G$. 
Ultimately, this leads to the following.

\begin{defn}\label{admissibleSymmetries}
Let $\mu:G\times M\to M$ be a Lie transformation group with Lie algebra $\bsy{\G}$ of infinitesimal generators leaving control system  (\ref{abstractControlSystem}) invariant and acting 
regularly and freely on $M$. We say that $G$ is a {\it control admissible} or simply {\it admissible} symmetry group if the function $t$ is invariant:  $\mu_g^*t=t$ for all $g\in G$ and the rank of the 
distribution $\pi_*\bsy{\G}$ is equal to  $\dim G$, where $\pi$ is the projection $\pi: M\to \B R\times X$, satisfying $\pi(t,\bsy{x},\bsy{u})=(t, \bsy{x})$. 
\end{defn}

If $(M, \CV)$ is the distribution associated to control system (\ref{abstractControlSystem}) then locally there are submanifolds $\mathrm{X}(M)$, the submanifold of states and $\mathrm{U}(M)$, the submanifold of controls such that $M=\B R\times\mathrm{X}(M)\times \mathrm{U}(M)$, where the factor $\B R$ is the time coordinate space. We will show that if $G$ is an admissible transformation group acting on $M=\B R\times \mathrm{X}(M)\times \mathrm{U}(M)$ then its elements are extended static feedback transformations.
\begin{thm}\label{WhenQuotientIsControl}
Let $\mu:G\times M\to M$ be an admissible Lie transformation group acting smoothly, regularly and freely on $M$ and leaving invariant the control system  $(M,\,\CV)$ defined by (\ref{abstractControlSystem}). 
Suppose $\dim G<\dim\mathrm{X}(M)$. Then locally the quotient $(M/G, \CV/G)$ is a control system in which $\dim\mathrm{X}(M/G)=\dim\mathrm{X}(M)-\dim G$ and $\dim\mathrm{U}(M/G)=\dim\mathrm{U}(M)$.
\end{thm}

\begin{proof}
The distribution $\CV$ has the form
$
\CV=\left\{\P t+\bsy{f}(t, \bsy{x}, \bsy{u})\P {\bsy{x}},\ \P {\bsy{u}}\right\} 
$
and any admissible symmetry of $\CV$ must preserve the subdistribution $\{\P {\bsy{u}}\}$. For if $v\in\bsy{\G}$ is admissible then $v(t)=0$ and hence $[v,\P {u_a}](t)=0$. Since $v$ is an infinitesimal symmetry, 
$[v,\P {u_a}]=\a T+\b_a\P {u_a}$,
for some functions $\a, \b_a$, where $T=\P t+\bsy{f}(t, \bsy{x}, \bsy{u})\P {\bsy{x}}$. We deduce that $\a=0$. 

Next if $v=\xi^i\P {x_i}+\eta^a\P {u_a}$ is an infinitesimal generator of an admissible symmetry of $\CV$ then
$
\ds \{\P {\bsy{u}}\}\ni [\P {u_a}, v]=\frac{\partial\xi^i}{\partial u_a}\P {x_i}+\frac{\partial\eta^b}{\partial u_a}\P {u_b}
$
which implies that $\ds \frac{\partial\xi^i}{\partial u_a}=0$. Hence, the corresponding infinitesimal generators of such an {\it admissible} symmetry group consists of vector fields of the form
$
\ds v=\bsy{\xi}(t,\bsy{x})\P {\bsy{x}}+\bsy{\eta}(t, \bsy{x}, \bsy{u})\P {\bsy{u}}.
$
If $\mu:G\times M\to M$ is the Lie transformation group with infinitesimal generators of the form $v$ then $\mu$ must have the general local form
$$
\mu_g(t, \bsy{x}, \bsy{u})=\left(t, \bsy{a}(t,\bsy{x},g),\ \bsy{b}(t, \bsy{x}, \bsy{u},g)\right)=(\bar{t},\ \bsy{\bar{x}},\ \bsy{\bar{u}})
$$
Recall that $r=\dim G<\dim\mathrm{X}(M)=n$. The action being admissible means that $\pi_*\bsy{\G}$ is a rank $r$ sub-bundle of $TM$, $r=\dim G$. If $\bsy{\G}=\{X_1,\ldots,X_r\}$, then
$$
\pi_*{X_i}={\frac{\partial a_\ell}{\partial g_i}{\P {x_\ell}}_\big|}_{g=\text{id}},\ \ 1\leq i\leq r.
$$ 
Since the $\pi_*X_i$ span a rank $r$ sub-bundle there is a subset $a_{i_1},\ a_{i_2},\ \ldots,\ a_{i_r}$ such that
$$
\text{det}\,{\frac{\partial  (a_{i_1},\ a_{i_2},\ \ldots,\ a_{i_r})}{\partial (g_1,\ \, g_2,\ \, \ldots,\ \, g_r)}_\big|}_{g=\text{id}}\neq 0.
$$
By the implicit function theorem, in a neighborhood of a point $(p, \text{id})\in \B R\times X\times G$ there are functions $g_j=\g_j(t,\bsy{x})$ such that $a_{i_s}(t,\bsy{x}, \bsy{\g}(t,\bsy{x}))=c_s$, 
 $1\leq s\leq r$ for some constants $c_s$. By the theory of equivariant moving frames \cite{FelsOlver}, 
the $n-r$ non-constant functions that remain among the components of
$ \bsy{a}(t, \bsy{x}, \bsy{\g}(t,\bsy{x}))$ and the $q$ non-constant functions $\bsy{b}(t,\bsy{x}, \bsy{u}, \bsy{\g}(t, \bsy{x}))$ together with $t$ span the $n+q+1-r$ invariants of the $G$-action. 
Setting $y_\ell$, $1\leq \ell\leq n-r$ equal to the non-constant functions among the $ \bsy{a}(t, \bsy{x}, \bsy{\g}(t,\bsy{x}))$ and $v_a$, \ $1\leq a\leq q$ equal to the functions $\bsy{b}(t,\bsy{x}, \bsy{u}, \bsy{\g}(t, \bsy{x}))$ produces the quotient map
$\mathbf{q}:M \to M/G$ in which local coordinates on $M/G$ have the form
\begin{equation}\label{quotientCoords}
t,\ \ y_\ell=y_\ell(t,\bsy{x}),\ \ v_a=v_a(t, \bsy{x}, \bsy{u}),\ \ 1\leq i\leq n-r,\ \ 1\leq a\leq q,
\end{equation}
and are components of the quotient map, $\mathbf{q}$. It follows that the quotient $\mathbf{q}_*\CV$ has the local form of control system
\begin{equation}\label{quotientControlSys}
\mathbf{q}_*\CV=\Big\{\P t+\sum_{\ell=1}^{n-r}\bar{f}^\ell(t, \bsy{y}, \bsy{v})\P {y_\ell},\ \P {v_1},\ \ldots, \P {v_q}\Big\}
\end{equation} 
for some functions $\bar{f}^\ell$, with the claimed dimensions of $\mathrm{X}(M/G)$ and $\mathrm{U}(M/G)$. 
\end{proof}

\begin{defn}
If $\mathbf{q}: M\to M/G$ is such that $\mathbf{q}_*\CV$ is a control system then we will say that $\mathbf{q}$ is a {\it control morphism}. 
\end{defn}

Not only do we wish to know when a symmetry group induces a control morphism $\mathbf{q}$ but also when $\mathbf{q}_*\CV$ is locally equivalent to a Brunovsky normal form by an (extended) static feedback transformation, directly from knowledge of the Lie algebra $\bsy{\G}$.

\begin{thm}\label{StatLinQuotients}
Let $(M, \CV=\{\P t+\bsy{f}(t,\bsy{x},\bsy{u})\P {\bsy{x}},\ \P {\bsy{u}}\})$ be a control system defining a totally regular sub-bundle of $TM,\ \dim M=n+q+1$ invariant under the free, regular and admissible action of Lie group $G$ on $M$ with Lie algebra $\bsy{\G}$. Suppose $(M, \CV\oplus\bsy{\G})$ is a relative Goursat bundle of derived length $k>1$ in which $\ch\CV^{(1)}_0=\CV\cap\bsy{\G}=\{0\}$. Then $\mathbf{q}:M\to M/G$ is 
a control morphism  and $\mathbf{q}_*\CV$ is locally equivalent to a partial prolongation $\mathcal{C}(\s)$ via local diffeomorphisms $\varphi: M/G\to J^\s(\B R, \B R^q)$,  
$\varphi_*\mathbf{q}_*\CV=\mathcal{C}(\s)$. A local equivalence $\varphi$ identifying $\mathbf{q}_*\CV$ and $\CC\langle \s\rangle$ can be chosen to be an (extended) static feedback transformation if and only if 
\begin{enumerate}
\item[1)] $\{\P {\bsy{u}}\}\subseteq\ch\widehat{\CV}^{(1)}_0$
\item[2)] $dt\in\mathrm{ ann}\,\ch\widehat{\CV}^{(k-1)}$  
\end{enumerate}
where $\widehat{\CV}=\CV\oplus\bsy{\G}$.
\end{thm}
\begin{proof} 
As in Proposition \ref{relativeGoursatProp1} we introduce a local trivialization $(\t_*M, \t_*\widehat{\CV})=(M/G\times Z, \CH\oplus\CZ)$ in which $\CZ=\t_*\ch\widehat{\CV}=\t_*\bsy{\G}$.
From Theorem \ref{WhenQuotientIsControl}, $\CH=\CV/G$ is a control system of the form (\ref{quotientControlSys}) with quotient map $\mathbf{q}:M\to M/G$ of the form (\ref{quotientCoords}). Suppose there is a extended static feeback transformation
$\varphi: M/G\to J^\s(\B R, \B R^s)$ such that $\varphi_*\CH= \CC(\s)$. Then by Theorem \ref{StatFeedbackLin}, we conclude that
$$
\{\P {\bsy{v}}\}\subseteq\ch\CH^{(1)}_0\ \ \text{and}\ \ d\bar{t}\in\mathrm{ ann}\,\ch\CH^{(k-1)}
$$
where $\bar{t}=(\pi_1\circ\t)(\bsy{p})$ and $k$ is the derived length of $\CH$ which agrees with the derived length of $\widehat{\CV}$. Here $\pi_1: \B R\times X\times U\to\B R$ is projection onto the first factor and $\bsy{p}\in \B R\times X\times U$ is a typical point. We have  $d\bar{t}(\t_*\bsy{\G})=d(\t^*\bar{t})(\bsy{\G})=dt(\bsy{\G})=0$ and hence $d\bar{t}\in\text{ann}\,\Big(\ch\CH^{(k-1)}\oplus\t_*\bsy{\G}\Big)=
\text{ann}\,\Big(\t_*\ch\widehat{\CV}^{(k-1)}\Big)$ and thus $dt\in\text{ann}\,\ch\widehat{\CV}^{(k-1)}$ which is item $2)$. 

Next we deduce $\{\P {\bsy{v}}\}=\t_*\{\P {\bsy{u}}\}\subseteq\ch\CH^{(1)}_0\oplus\t_*\bsy{\G}=\t_*\ch\widehat{\CV}^{(1)}_0$ from which item $1)$ follows.

Conversely, suppose $1)$ and $2)$ hold. Since $\widehat{\CV}$ is a relative Goursat bundle, by Theorem  \ref{WhenQuotientIsControl}, there is a local diffeomorphism $\varphi:M/G\to J^\s(\B R,\B R^q)$ such that 
$\varphi_*\big(\CV/G\big)=\CC\langle\s\rangle$ some integer $q$ and signature $\s$.  Let $\t:M\to M$ be defined as in Proposition \ref{relativeGoursatProp1}. From $1)$, we have $\t_*\{\P {\bsy{u}}\}\subseteq\t_*\ch\widehat{\CV}^{(1)}_0$ which implies
$\{\P {\bsy{v}}\}\subseteq\ch\CH^{(1)}_0$. Since $(\pi_1\circ\t)(\bsy{p})=t$, from $2)$ we deduce that $d\bar{t}\in\text{ann}\,\t_*\ch\widehat{\CV}^{(k-1)}$ which implies that $dt\in\text{ann}\,\ch\CH^{(k-1)}$. By Theorem  \ref{StatFeedbackLin}, we conclude that $\varphi$ can be chosen to be an extended static feedback transformation.
\end{proof}

\subsection{A non-flat control system with a flat quotient}\label{non-flat}

There are few results that provide us with classes of non-flat control systems with more than a single input. One source of 2-input non-flat control systems is described by the following result of Martin and Rouchon \cite{MartinRouchon94}.

\begin{thm}[\cite{MartinRouchon94}]\label{MartinRouchonThm}
A driftless system $\dot{x}=f_1(x)u_1+f_2(x)u_2$ in $n$ states $x$ and two inputs $u$ is flat if and only if the elements of the derived flag of $E=\{f_1, f_2\}$ satisfy
$$
\dim E^{(k)}=\dim E^{(k-1)}+1, \ \ E^{(0)}=E,\ \ k=1,\ldots,n-2. 
$$
\end{thm} 
Notice that the condition imposed on $E$ implies that it is equivalent to the contact system 
$\mathcal{C}^{n-2}(\mathbb{R},\mathbb{R})$ on $J^{n-2}(\mathbb{R},\mathbb{R})$ via a local diffeomorphism of state space $X$, by the Goursat normal form.

Granting this, a particularly elegant class of non-flat control systems in 5 states and 2 control can be constructed by taking $E$ to be a generic 2-plane distribution on $\B R^5$ whose growth vector is 
$[2,\ 3,\ 5]$. Such distributions have been classified, \cite{Cartan10}. The most symmetric of these has local realisation 
$
{E}=\Big\{\P {x_1}+x_3\P {x_2}+x_5\P {x_3}+x_5^2\P {x_4},\ \P {x_5}\Big\}=\{f_1,\ f_2\}.
$
The control system corresponding to the driftless system of the Martin-Rouchon theorem is therefore
\begin{equation}\label{nonflatRouchon}
\CV=\Big\{\P t+u_1f_1+u_2f_2,\ \P {u_1},\ \P {u_2}\Big\}.
\end{equation}
Its refined derived type is
\begin{equation}\label{derivedTypeRouchon}
\mathfrak{d}_r(\CV)=[\, [3,\, 0], [5,\, 2,\, 3],\, [6,\, 3,\, 3],\, [8,\, 8]\, ].
\end{equation}
Since $\mathfrak{d}_r(\CV)$ does not satisfy Proposition \ref{derivedtype}, this verifies the fact that $\CV$ is not linearizable by {\it any} local diffeomorphism of $\B R^8$. Because of its non-flatness there can be no partial prolongation of $\CV$ which is static feedback linearizable.\ 
Let us next investigate the role of symmetry in identifying {\it flat subsystems} of (\ref{nonflatRouchon}). 
A subalgebra of the symmetry algebra of $\CV$ is spanned by the vector fields
$
\text{Sym}(\CV)=\{X_1,\ldots, X_5\}
$
where 
$$
\begin{aligned}
&X_1=\frac{1}{2}x_1^2\P {x_2}+x_1\P {x_3}+2x_3\P {x_4}+\P {x_5},\ \ X_2=x_1\P {x_2}+\P {x_3},\ \ X_3=\P {x_4},\ \ X_4=\P {x_1},\ \ X_5=\P {x_2},
\end{aligned}
$$
and these are control admissible since they are state-space symmetries. Let us for instance consider the quotient of $(\B R^8,\CV)$ by the local Lie transformation group generated by the abelian subalgebra $\bsy{\G}_0=\{X_1,X_3\}$. The action is regular, for instance, in a neighbourhood of $0\in\B R^8$. Invoking Theorem \ref{quotientCheck}, we  check that the refined derived type of $\widehat{\CV}:=\CV\oplus\bsy{\G}_0$ is
$
\mathfrak{d}_r(\widehat{\CV})=[\,[5, 2], [7, 4, 5], [8, 8]\,]
$  
which satisfies the constraints of Proposition \ref{derivedtype} with signature $\s=\text{decel}\,\widehat{\CV}=\langle 1,1\rangle$. The only nontrivial fundamental bundle is $\Xi^{(1)}_0/\Xi^{(1)}$ while calculation shows that this is equal to $\{dt\}$ and $\Xi^{(1)}_0$ is integrable. These observations guarantee, by Definition \ref{relativeGoursatBundle}, that $\widehat{\CV}$ is a relative Goursat bundle of signature $\s=\langle 1,1\rangle$. By Theorem \ref{quotientCheck}, the quotient $\CV/\bsy{\G}_0$ is locally diffeomorphic to the partial prolongation $\mathcal{C}\langle 1, 1\rangle$. 

We now use Theorem \ref{StatFeedbackLin} to check for the existence of a static feedback linearization of $\CV/\bsy{\G}_0$. We find that (with derived length $k=2$), $dt\notin \Xi^{(1)}(\widehat{\CV})$ and hence no static feedback transformation exists. In summary, we have the following facts concerning control system (\ref{nonflatRouchon}):
\begin{itemize}
\item System (\ref{nonflatRouchon}) is not flat by Theorem \ref{MartinRouchonThm}.
\item The quotient $\CV/\bsy{\G}_0$ of (\ref{nonflatRouchon}) is linearizable but {\it not} via a static feedback transformation by Theorem \ref{StatFeedbackLin}.
\end{itemize} 
It turns out, in fact, that $\CV/\bsy{\G}_0$ is flat. We prove this by showing that some prolongation of it {\it is} static feedback linearizable.
Indeed,  prolong $\CV$ as follows
$
\text{pr}\,\CV=\Big\{\P t+u_1f_1+u_2f_2+v_1\P {u_1},\ \P {v_1},\ \P {u_2}\Big\}.
$
Again, as per Theorem \ref{quotientCheck}, we study the augmented distribution $\text{pr}\, \widehat{\CV}:=\text{pr}\,\CV\oplus\bsy{\G}_0$ and calculate refined derived type to be
$
\mathfrak{d}_r\left(\text{pr}\,\widehat{\CV}\right)=[\,[5, 2], [7, 4, 4], [9, 9]\,].
$
This satisfies item 1. of Definition \ref{relativeGoursatBundle} with signature $\s=\text{decel}\left(\text{pr}\, \widehat{\CV}\right)=\langle 0, 2\rangle$. Since the final entry of $\s$ is greater than 1, we complete the check that $\text{pr}\, \widehat{\CV}$ is a relative Goursat bundle by verifying that its resolvent bundle is integrable. Indeed, we get $\ch\text{pr}\,\widehat{\CV}^{(1)}=\{\P {x_4}, \P {u_2}, \P {v_1}, X_1\}$ and calculate that the rank 6 resolvent bundle is integrable. 
Therefore we deduce that the prolongation $\text{pr}\,\CV$ of $\CV$ has a static feedback linearizable quotient by $\bsy{\G}_0$. A general result can be established which proves that  $\CV/\bsy{\G}_0$ is flat.  We will not elaborate on this here but rather verify the result explicitly.  

The Lie transformation group $G_0$ generated by $\bsy{\G}_0$ acts freely on $\B R^8$. The first integrals of $\bsy{\G}_0$ are spanned by
$
\text{inv}\,\bsy{\G}_0=\{t,\ u_1,\ u_2,\ x_1,\ x_1x_3-2x_2,\ x_1^2x_5-2x_2\}.
$
Implementing Theorem \ref{mainThm}, set $w_1=x_1,\ w_2=x_1x_3-2x_2,\ w_3=x_1^2x_5-2x_2$, and obtain
\begin{equation}\label{MartinQu}
\CH:=\CV/\bsy{\G}_0=\Big\{\P t+u_1\Big(\P {w_1}-\frac{w_2-w_3}{w_1}\big(\P {w_2}+2\P {w_3}\big)\Big)+u_2w_1^2\P {w_3},\ \P {u_1},\ \P {u_2}\Big\}.
\end{equation}

Using Theorem \ref{GGNF} it can be shown that (\ref{MartinQu}) is locally diffeomorphic to the contact distribition $\mathcal{C}\langle 1,1\rangle$, as predicted by Theorem \ref{quotientCheck}. That is, the Brunovsky normal form
$
\dot{y}_1=u_1,\ \dot{z}_1=z_2,\ \dot{z}_2=u_2.
$
This is in stark contrast to the geometric structure of $\CV$ which is not locally diffeomorphic to a Brunovsky normal form having refined derived type
(\ref{derivedTypeRouchon}) and therefore fails to satisfy the necessary constraints of Proposition \ref{derivedtype}. However, it is easily checked that a prolongation of $\CH$ in direction $\P {u_1}$ results in a distribution locally equivalent to $\CC\langle 0,2\rangle$ by a static feedback transformation. This verifies that $\CH$ is flat, as predicted.  Thus, if $\bsy{\bar{\o}}=\ker\CH$, we have 
$\bsy{\o}\supset\mathbf{q}^*\bsy{\bar{\o}}=\big\{x_1\o_1-\o_3,\ \frac{1}{2}x_1^2\o_1-\o_4,\ \o_5\big\}$ is a subsystem of $\bsy{\o}=\{\o_1,\ldots,\o_5\}=\mathrm{ann}\CV$ which projects to a flat control system on 
$M/G_0$, $\mathbf{q}:M\to M/G_0$ being the quotient map.



\section{Application: Trajectory Planning for Under-Actuated Ships}\label{shipSection}

We now take up a major application of the forgoing results by examining the role of symmetry reduction in trajectory planning for under-actuated marine vessels. A textbook account of this dynamical system is given in \cite{fossen}. The full 6 degrees-of-freedom under-actuated ship dynamical system which takes hyrodynamic damping and the effects of the Earth's rotation into account is usually treated by neglecting the heave, pitch and roll degrees of freedom leaving surge, yaw and sway as the most significant motions in building ship guidance control systems.   

The simplest such control system, discussed in a number of recent works, for instance \cite{ramirez}, is 
\begin{equation}\label{shipcontrol}
\begin{aligned}
&\dot{x}=u_1\cos\th-z\sin\th, \ \ 
\dot{y}=u_1\sin\th+z\cos\th,\ \ 
\dot{\th}=u_2,\ \ 
\dot{z}=-\g u_1u_2-\b z.
\end{aligned}
\end{equation}

Here $u_1$ controls the forward velocity or {\it surge}; $u_2$ controls the angular velocity of the ship about its midpoint axis, {\it yaw}. Variable $z$ represents the velocity with which the ship is displaced in a direction perpendicular to its longitudinal axis, called {\it sway}; see { Fig.} 1.   Besides neglecting pitch, heave and roll, other reasonable simplifications have been made in order to arrive at control system (\ref{shipcontrol}).

\vskip 0 pt

\begin{center}
\begin{tikzpicture}[scale=0.9]
\filldraw[fill=green!5](-5,-2) rectangle (5.5,4.3);
\filldraw[fill=blue!15][rotate=-45](0,0) rectangle (2,4);
\draw  [rotate=-45] (0,0) rectangle (2,4);
\filldraw[fill=blue!15][rotate=-45](0,4) .. controls  (1,6) .. (2,4);
\draw[rotate=-45](0,4) .. controls  (1,6) .. (2,4);
\draw[rotate=-45](2,0) arc (0:-180:1cm); 
\filldraw[rotate=-45][fill=blue!15](2,0) arc (0:-180:1cm); 

\filldraw[xshift=45 pt,yshift=36 pt][blue](0.5,-0.5) circle (1mm);
\draw[xshift=45 pt,yshift=36 pt](0.7,-0.5)--(1.9,-0.5);
\draw[xshift=30 pt,yshift=30 pt][thick,->](-1,1)--(-2.0,2.0);
\draw[xshift=30 pt,yshift=30 pt](-1.3,1.5)node[above]{$z$};
\draw[xshift=45 pt,yshift=36 pt](0.6,-0.4)--(1.5,0.5);
\draw[xshift=45 pt,yshift=36 pt](1.3,-0.2) node{$\th$};
\draw[xshift=50 pt,yshift=30 pt][thick,->](-6,-2)--(-6,0);
\draw[xshift=50 pt,yshift=30 pt][thick,->](-6,-2)--(-4,-2);
\draw[xshift=50 pt,yshift=30 pt](-6.0,-2.0)node[below]{$O$};
\draw[xshift=50 pt,yshift=30 pt](-6.2,0)node{$y$};
\draw[xshift=50 pt,yshift=30 pt](-4,-2)node[below]{$x$};
\draw[xshift=45 pt,yshift=36 pt](0.5,-0.5)node[below]{$(x,y)$};
\end{tikzpicture}
\end{center}
\begin{center}
\vskip 10 pt
{\small \bf Figure 1: The under-actuated ship.}
\end{center}
\vskip 15 pt
For simplicity the constants $\b$ and $\g$ in (\ref{shipcontrol}) have been set to unity in this section. To study (\ref{shipcontrol}) we form the Pfaffian system

\begin{equation}\label{shipPfaff}
\begin{aligned}
\bsy{\o}=&\Big\{dx-(u_1\cos\th-z\sin\th)\,dt,\ dz-(u_1\sin\th+z\cos\th)\,dt,\ d\th-u_2\,dt,\ dz+(u_1u_2+z)\,dt  \Big\}\cr
        =&\Big\{\o^1,\ \ \ \o^2,\ \ \ \o^3,\ \ \ \o^4\Big\}.
\end{aligned}
\end{equation} 
The refined derived type of 
$
\CE=\ker\bsy{\o}
$
is
$$
\mathfrak{d}_r(\CE)=[\,[3,0], [5,0,0],[7,7]\,].
$$
This proves, by Theorem \ref{GGNF}, that (\ref{shipPfaff}) is not linearizable -- there is {\it no equivalence} to a Brunovsky normal form.   
On the other hand, we will prove that (\ref{shipPfaff}) is {\it flat}.\footnote{The authors in \cite{ramirez} state that (\ref{shipPfaff}) is {\it not} flat. It is shown later in this section that a prolongation of (\ref{shipPfaff}) is static feedback linearizable implying that (\ref{shipPfaff}) is indeed flat.} However, it must be stressed that its flatness {\it does not} facilitate trajectory planning in an obvious way, as we shall see.

\subsection{The control admissible symmetries.} We conjecture that the symmetry group of (\ref{shipPfaff}) is infinite-dimensional and in full generality possibly very complicated. However, a significant realization of this work is that the full group of symmetries is not required  for control purposes, but only the control admissible symmetries. A calculation shows that the Lie algebra of infinitesimal control admissible symmetries
$\bsy{\Upsilon }$ is spanned by $X_1,\ldots, X_8$, where 

\begin{equation}\label{shipSyms}
\begin{aligned}
&X_1=-y\P x+x\P y+\P {\th},\ \ \  \ \ \  
X_2= x\P x+y\P y+z\P z+u_1\P {u_1},\cr
&X_3=e^{-t}\big(\P y-\cos \theta\P z-\sin \theta\P {u_1}\big),\ \ \ \ \ 
X_4=e^{-t}\big(\P x-\sin \theta \P z+\cos \theta \P {u_1}\big),\cr
&X_5=-x\P x+y\P y+\sin2\th\P {\th}+z\cos2\th\P z-(u_1\cos 2\th-2z\sin 2\th)\P {u_1}+2u_2\cos 2\th\P {u_2}\cr
&X_6=y\P x+x\P y+\cos 2\th\P \th-z\sin 2\th\P z+(u_1\sin 2\th+2z\cos 2\th)\P {u_1}-2u_2\sin 2\th\P {u_2}\cr
&X_7=\P y, \ \ \ \ X_8=\P x.
\end{aligned}
\end{equation}
The Levi decomposition $\bsy{\Upsilon}=\mathbf{S}\ltimes \mathbf{R}$ has semisimple subalgebra $\mathbf{S}=\{X_1, X_5, X_6\}$ isomorphic to $\mathfrak{sl}(2)$ while its radical is $\mathbf{R}=\{X_2, X_3, X_4, X_7, X_8\}$. In general, the Lie algebra of infinitesimal admissible symmetries forms a subalgebra of the Lie algebra of all infinitesimal symmetries. The usually considered \cite{GrizzleMarcus}, \cite{Elkin} state space symmetries $\{X_1, X_7, X_8\}$ form a comparatively trivial subalgebra of the control admissible symmetries, expressing the obvious geometry of the ship configuration. 

Even if we specialize to subalgebra $\mathbf{R}$, there is a lot of choice  as to the subalgebra of $\mathbf{R}$ to be used for symmetry reduction.  In this paper, due to space limitations, we shall not dwell on the interesting problem of classifying the control morphisms for (\ref{shipPfaff}). The immediate goal is to show how the theory of this paper explicitly solves the problem of trajectory planning.

We shall investigate the 2-dimensional abelian Lie subalgebra
$$
\bsy{\G}=\Big\{e^{-t}\left(\P x-\sin\th\P z+\cos\th\P {u_1}\right),\ \ \P x\Big\}=\Big\{X_4,\ \ X_8\Big\}\subset\mathbf{R}.
$$
At some point, we shall require the Lie group action generated by the 1-parameter subgroups of $\bsy{\G}$ on the 7-dimensional space $M$ upon which $\bsy{\o}$ is defined and which has coordinates $\bsy{m}=(t, x, y, z, \th, u_1, u_2)$. Since $\bsy{\G}$ is abelian the Lie group in question is just $G=\mathbb{R}^2$.  We will continue to denote it by $G$. The action $\mu:G\times M\to M$ is explicitly
\begin{equation}\label{action}
\bsy{\widetilde{m}}=\mu(\e_1,\e_2,\bsy{m})=\Big(t,\ \ x-\e_1e^{-t}+\e_2,\ \ y,\ \ \th,\   \ z-\e_1e^{-t}\sin\th,\ \ u_1+\e_1e^{-t}\cos\th,\ \ u_2\Big)
\end{equation}
Here $\e_1,\ \e_2$ are coordinates on $G$. Recall that $\mu$ constitutes a Lie transformation group of symmetries of $\bsy{\o}$ in the sense that
$
\mu_{\bsy{\e}}^*\bsy{\o}\subseteq \bsy{\o}
$
where $\mu_{\bsy{\e}}(\bsy{m})=\mu(\bsy{\e},\bsy{m})$,\ \ $\bsy{\e}=(\e_1,\ \e_2)$. The action of $G$ is regular and free.  
Furthermore the action is {\it transverse} to $\bsy{\o}$: $\bsy{\G}\cap\CE=0$, where
$$
\CE=\Big\{\P t+(u_1\cos\th-z\sin\th)\P x+(u_1\sin\th+z\cos\th)\P y+u_2\P {\th}-(u_1u_2+z)\P z,\ \P {u_1},\ \P {u_2}\Big\}.
$$
It is also transverse to $\CE^{(1)}$ so that by \cite{AF}, (Theorem 5.1), its symmetry reduction by this $G$-action is Pfaffian. Finally, the hypotheses of Theorem \ref{mainThm} are satisfied. We will now use this to solve the trajectory planning problem for (\ref{shipcontrol}). 

\subsection{Quotient of the control EDS}

For this application it turns out to be easy to construct the quotient $\CE/G$ and then check that it is static feedback linearizable. However, in general this will not be the case and accordingly we will first show how to perform this check using the algorithmic results of section \ref{contactFlat}.  This implies that we invoke Theorems \ref{quotientCheck} and \ref{StatFeedbackLin} to show that the quotient $\CE/G$ of $\CE$ by the $\B R^2$-action generated by $\bsy{\G}$ is static feedback linearizable. With 
$\widehat{\CE}=\CE\oplus\bsy{\G}$, we find that 
$$
\mathfrak{d}_r(\widehat{\CE})=[\,[5,2], [7,7]\,]
$$ 
which satisfies item 1. of Definition \ref{relativeGoursatBundle} of a relative Goursat bundle with signature 
$$
\text{decel}\,\widehat{\CE}=\langle 2\rangle.
$$ 
From $\text{decel}\,\widehat{\CE}$, we deduce that only the resolvent need be considered to complete the check of Definition \ref{relativeGoursatBundle} which is easily shown to be 
$R(\widehat{\CE})=\{\P {u_1}, \P {u_2}\}\oplus\bsy{\G}$. This is plainly integrable. Hence by Theorem \ref{quotientCheck}, the quotient $\CE/G$ is locally equivalent to $\mathcal{C}\langle 2\rangle$.  
Finally, we check for the existence of static feedback equivalences. In this case the derived length is $k=1$ so Theorem \ref{StatFeedbackLin} doesn't apply. It is possible to deal with this case also, however we will instead verify that static feedback equivalences exist by direct calculation.

We can do this from knowledge of the semi-basic forms as in $\cite{AF}$ or else by computing $\mathbf{q}_*\CE$, where $\mathbf{q}: M\to M/G$ is the quotient map.  
Now the invariants of the action are easily shown to be
$
\{t,\ y,\ \theta,\ u_2,\ u_1+z\cot\th\}
$
or else we could take the set
$
\{t,\ y,\ \theta,\ u_2,\ u_1\sin\th+z\cos\th\}.
$
Using the former, the quotient map $\mathbf{q}:M\to M/G$ is taken to be
$$
\mathbf{q}(t,x,y,z,\th,u_1,u_2)=(t=t,\ w_1=y,\ w_2=\th,\ w_3=u_2,\ w_4=u_1+z\cot\th)
$$ 
where $t, w_1, w_2, w_3, w_4$ are labels for the local coordinates on $M/G$. 
It can be checked that a cross-section $\s:M/G\to M$ for $\mathbf{q}$ is given by
$$
\s(t, w_1, w_2, w_3, w_4)=\big(t=t,\ x=0,\ y=w_1,\ \th=w_2,\ z=0,\ u_1=w_4,\ u_2=w_3).
$$
We get the quotient of $\bsy{\o}$ to be
$$
\bsy{\bar{\o}}=\Big\{dw_2-w_3dt,\ dw_1-w_4\sin(w_2)dt\Big\}.
$$
This Pfaffian system is particularly easy to integrate and is clearly linearizable by static feedback transformations. The integral manifolds are expressible in terms of two functions $p(t), q(t)$ and their derivatives
$$
\bar{s}(t)=\left(t,\ w_1=p(t),\ w_2=q(t),\ w_3=\dot{q}(t),\ w_4=\frac{\dot{p}(t)}{\sin\big(q(t)\big)}\right).
$$
This is the general solution of the quotient $\mathbf{q}_*\CE=\CE/G$. 

\subsection{General solution of the ship control system}

We now implement Theorem \ref{mainThm} to construct all the trajectories of (\ref{shipPfaff}). Effectively, the theorem permits us to suppose that the trajectories of (\ref{shipPfaff}) have the form
$$
s(t)=\mu\big(g(t),(\s\circ\bar{s})(t)\big)
$$ 
for some function $g:\mathbb{R}\to G\simeq \mathbb{R}^2$. If $g(t)=(g_1(t),\ g_2(t))$ then we find that $s^*\bsy{\o}=0$ implies that $g$ satisfies the completely integrable system
$$
\begin{aligned}
&\dot{g}_1(t)=-e^t\dot{p}(t)\frac{d}{dt}\Big(\cot\big(q(t)\big)\Big),\ \ 
\dot{g}_2(t)=\dot{p}(t)\left(\cot\big(q(t)\big)-\frac{d}{dt}\Big(\cot\big(q(t)\Big)\right)
\end{aligned}
$$
and hence
$$
g(t)=\left(\int {-e^t\dot{p}(t)\frac{d}{dt}\Big(\cot\big(q(t)\big)\Big)}\,dt,\ \ \int{\dot{p}(t)\left(\cot\big(q(t)\big)-\frac{d}{dt}\Big(\cot\big(q(t)\Big)\right)}\,dt\right).
$$
We then obtain all the trajectories in the form
$$
\begin{aligned}
&x = -e^{-t}g_1(t)+g_2(t),\ \  y = p(t),\ \  \th = q(t),\ \  z = -e^{-t}g_1(t)\sin\big(q(t)\big), \cr
&u_1 = e^{-t}g_1(t)\cos\big(q(t)\big)+\frac{\dot{p}(t)}{\sin\big(q(t)\big)},\ \  u_2 = \dot{q}(t).
\end{aligned}
$$

For instance if we fix $\ds  \th=\frac{\pi}{4}$ and $y=p(t)$, arbitrary, then
$$
x(t)=-\frac{\pi}{4}e^{-t}+p(t),\ \ z(t) = -\frac{\sqrt{2}}{8\pi}e^{-t}, \ \ u_1(t)=\frac{\sqrt{2}}{8}\left(e^{-t}\pi+8\dot{p}(t)\right),\ \ u_2(t)=0.
$$

\subsection{Trajectory planning for the under-actuated ship}

We wish to prescribe a surface trajectory $\bsy{C}$ given by $\bsy{x}(t)=(x(t),\,y(t))$. Suppose we choose the parametrisation of $\bsy{C}$ in the form $\bsy{x}(t)=(x(t),t)$. That is, $\bsy{C}$ is a graph over the $y$-axis. Looking at the formulas for the trajectories, we find that
$$
g_1(t)=-\int{e^t\dot{Q}}\,dt,\ \ g_2(t)=\int{(Q-\dot{Q})}\,dt
$$ 
where
$$
Q(t)=\cot\big(q(t)\big).
$$
Integration by parts gives $\ds g_1(t)=-\left(e^tQ-\int{e^tQ}\,dt\right)$ and hence
$$
x(t)=e^{-t}\left(e^tQ-\int{e^tQ}\,dt\right)+\int{Q}\,dt-Q
$$
and we wish to express $Q$ in terms of $x(t)$. Multiplying both sides by $e^t$ and differentiating gives
$$
\frac{d}{dt}\Big(x(t)e^t\Big)=e^t\int{Q}\,dt.
$$
It easily follows that
$$
q(t)=\cot^{-1}\left(\frac{d}{dt}\left(e^{-t}\frac{d}{dt}\Big(e^tx(t)\Big)\right)\right)=\cot^{-1}\big(\dot{x}(t)+\ddot{x}(t)\big).
$$
Thus every trajectory of the form $\bsy{x}(t)=(x(t),t)$ can be prescribed exactly. 
Indeed if we prescribe the surface path $\bsy{x}(t)=(x(t),t)$ then it is straightforward to derive, from the above argument, that the system trajectories are given by
$$
\begin{aligned}
&\theta(t)=\cot^{-1}(\dot{x}+\ddot{x}), \ \ z(t)=\frac{\ddot{x}}{\sqrt{1+(\dot{x}+\ddot{x})^2}},\ \ 
u_1(t)=\frac{1+\dot{x}(\dot{x}+\ddot{x})}{\sqrt{1+(\dot{x}+\ddot{x})^2}},\ \ u_2(t)=\frac{\ddot{x}+\dddot{x}}{1+(\dot{x}+\ddot{x})^2}.
\end{aligned}
$$
The function $g:\mathbb{R}^2\to G$ is given by
$$
g(t)=\left(-e^t\,\ddot{x},\ x-\ddot{x}\right).
$$
One can check that the map $s:\mathbb{R}\to M$ given by
$$
s(t)=\left(t,\ x(t),\ t,\ \theta(t),\ z(t),\ u_1(t),\ u_2(t)\right)=(t,\ x,\ y,\ \theta,\ z,\ u_1,\ u_2) 
$$
satisfies $s^*\bsy{\o}=0$ and is therefore an integral submanifold of the control system EDS $\bsy{\o}$.

On the face of it this may seem to imply that only trajectories which are graphs over the $y$-axis can be planned in this way. But this is not so because we can toggle the physical definition of $x$ and $y$ in a given real situation. One can also use other symmetries (\ref{shipSyms}) and symmetry reductions for this purpose. Thus every surface trajectory can be explicitly and exactly planned.

\subsection{Examples of trajectory planning}
\subsubsection{Straight line path.} If $\bsy{x}(t)=(\l t+\mu,\ t)$ then the system trajectories are
$$
\theta(t)=\cot^{-1}(\lambda),\ z(t) = 0,\ u_1(t) = \sqrt{1+\lambda^2},\ u_2(t) = 0.
$$

\subsubsection{Parabolic path.} If $\ds\bsy{x}(t)=\left(\frac{1}{2}t^2,\ t\right)$ then the system trajectories are
$$
\theta(t)=\cot^{-1}(t+1),\ z(t) = \frac{1}{\sqrt{2+2t+t^2}},\ u_1(t) = \frac{1+t+t^2}{\sqrt{2+2t+t^2}},\ u_2(t) = -\frac{1}{2+2t+t^2}
$$

\hskip -25 pt \includegraphics[width=5.5cm,height=4cm,scale=0.6]{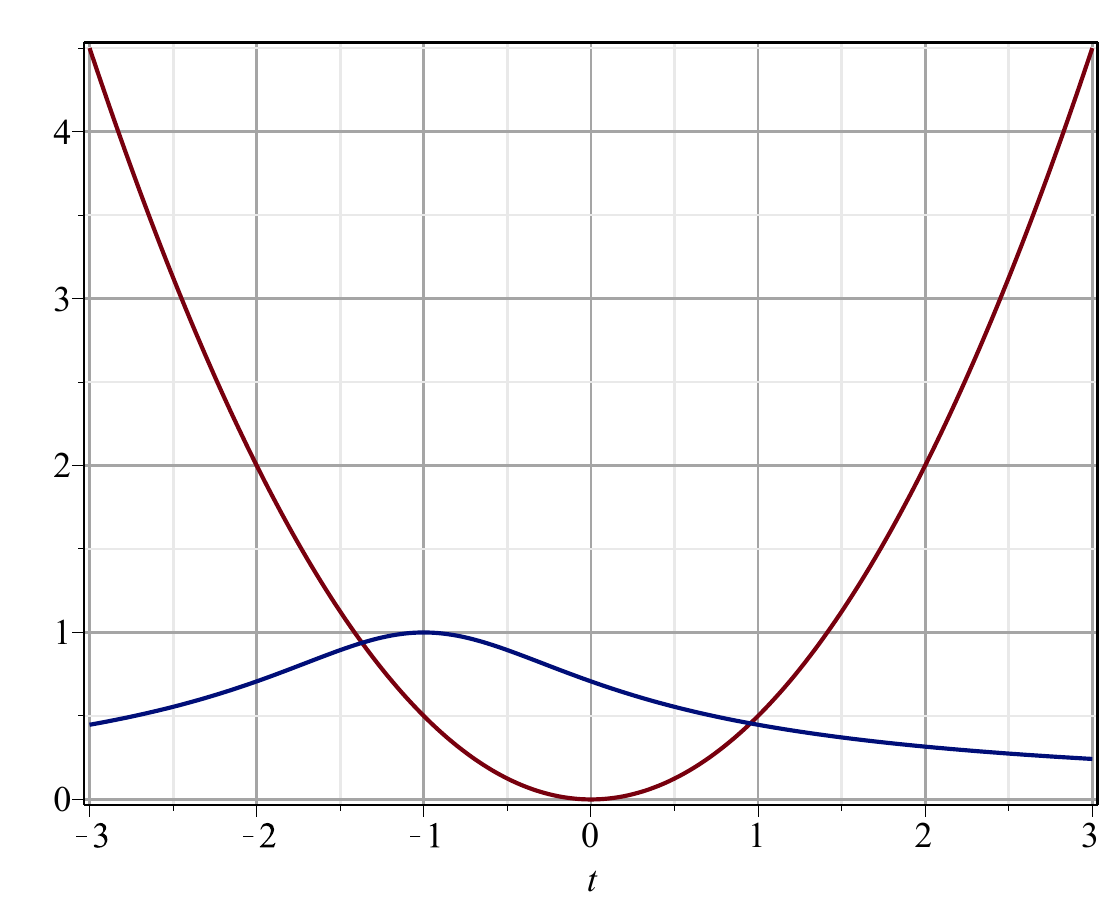} \includegraphics[height=4cm,width=5.5cm,scale=0.6]{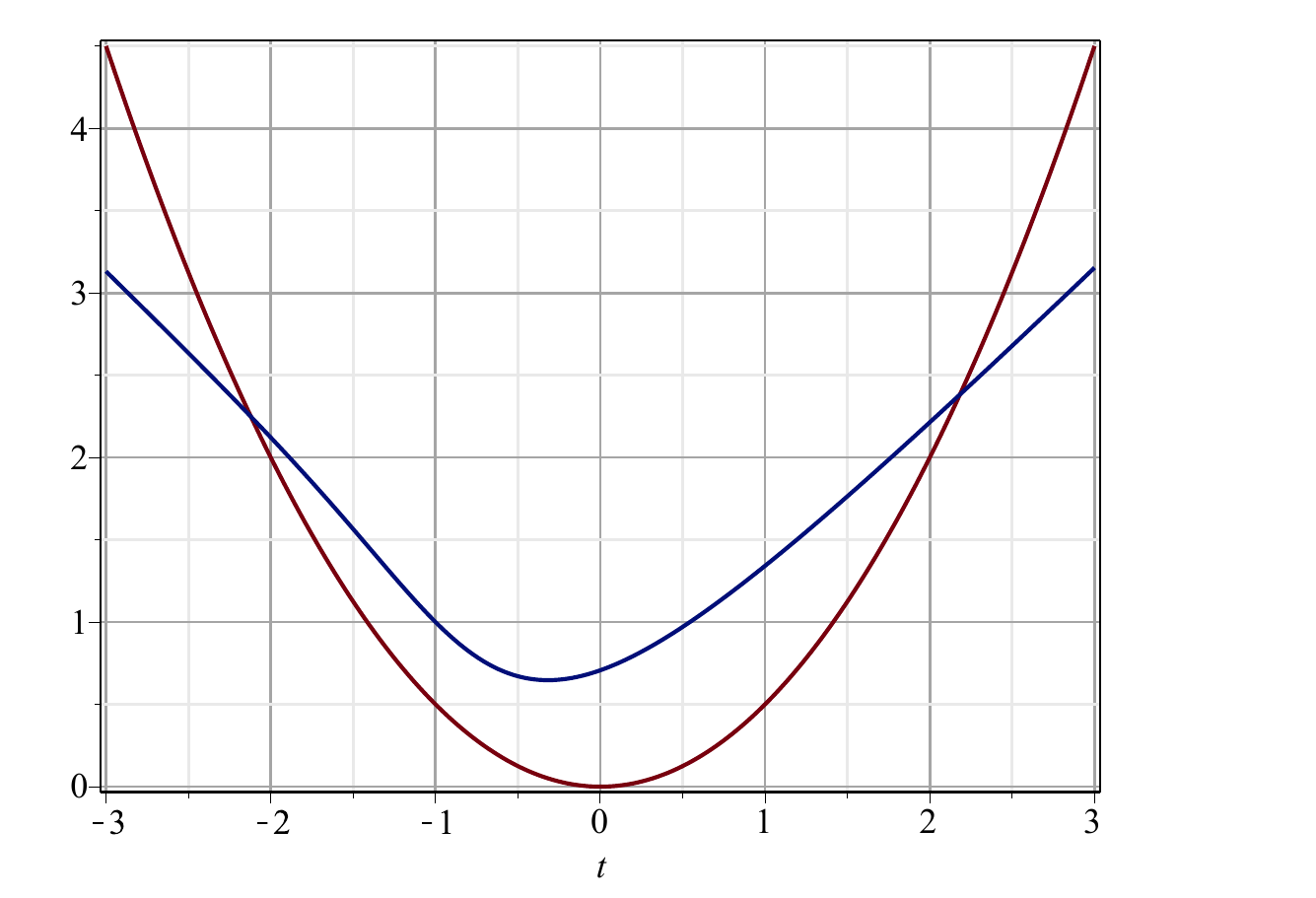} \includegraphics[height=4cm,width=5.5cm,scale=0.6]{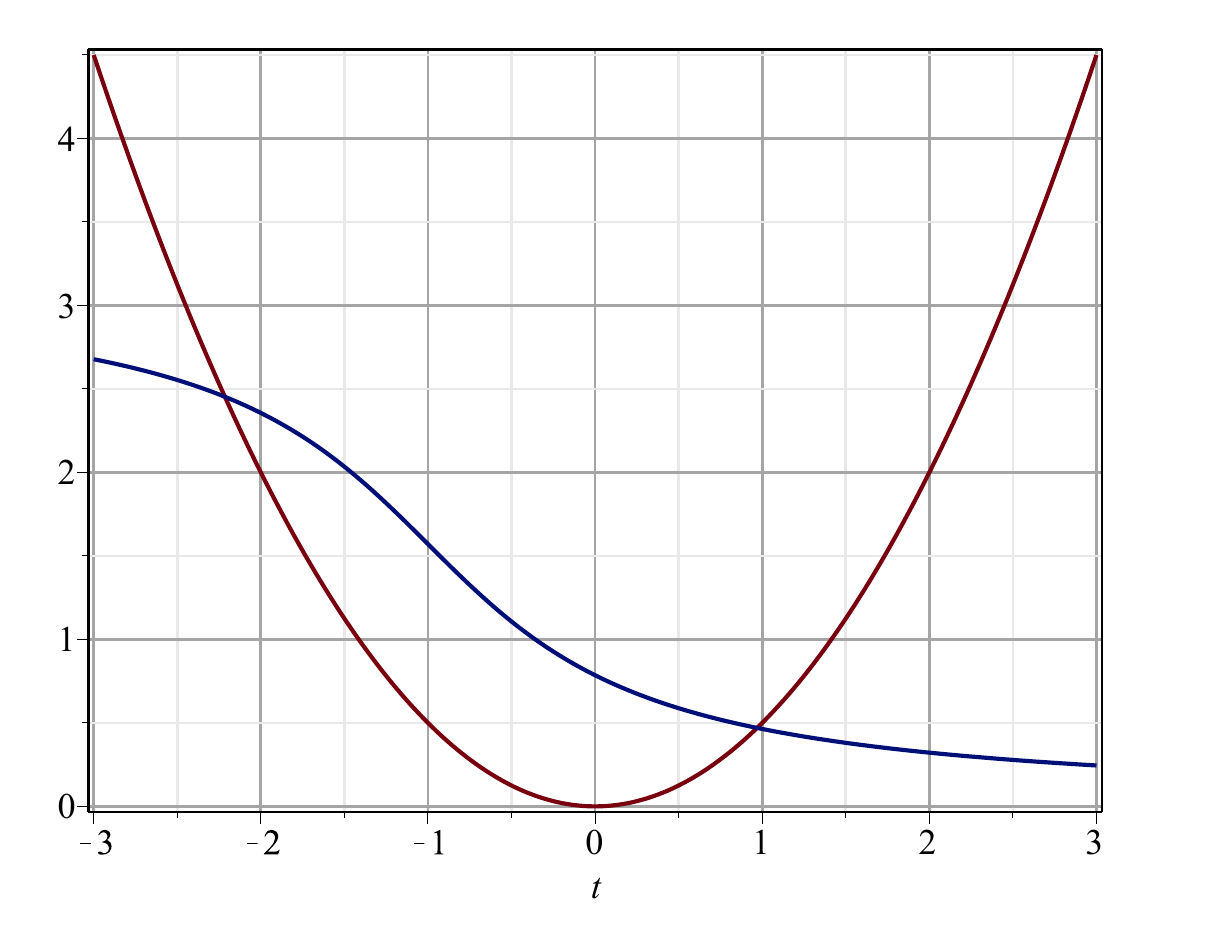}

\begin{center}
{\small \bf Fig. 2. Graphs of $\bsy{x}\ \&\ \bsy{z}$,\ \ $\bsy{x}\ \&\ \bsy{u_1}$,\ \ $\bsy{x}\ \&\ \bsy{\theta}$, respectively, for the parabolic path $\ds \bsy{x}=\bsy{\frac{\bsy{1}}{\bsy{2}}t^2}$}
\end{center}

\subsubsection{Circular path} If $\bsy{x}(t)=(\pm\sqrt{1-t^2},\ t)$ then the system trajectories involve complicated expressions. For instance for the upper and lower semi-circles respectively we get
$$
\theta_+(t)=\frac{\pi}{2}-\tan^{-1}\left(\frac{t^3-t-1}{(1-t^2)^{3/2}}\right),\ \ \ 
\theta_-(t)=\frac{\pi}{2}+\tan^{-1}\left(\frac{t^3-t-1}{(1-t^2)^{3/2}}\right).
$$
Of interest are the graphs of the dynamical variables in comparision to the surface path
\vskip 15 pt
\hskip 15 pt\includegraphics[height=6cm,width=4.5cm]{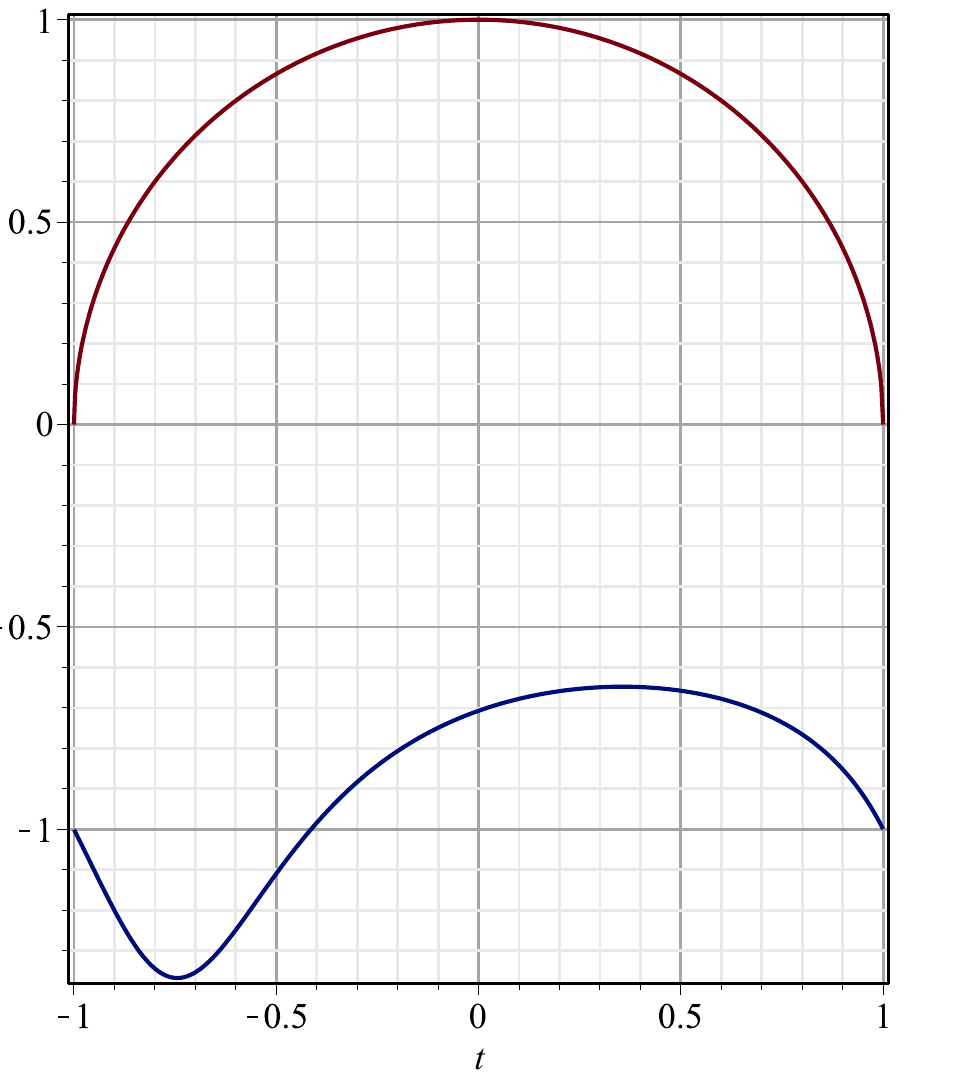} \  \includegraphics[width=3.5cm,height=6cm]{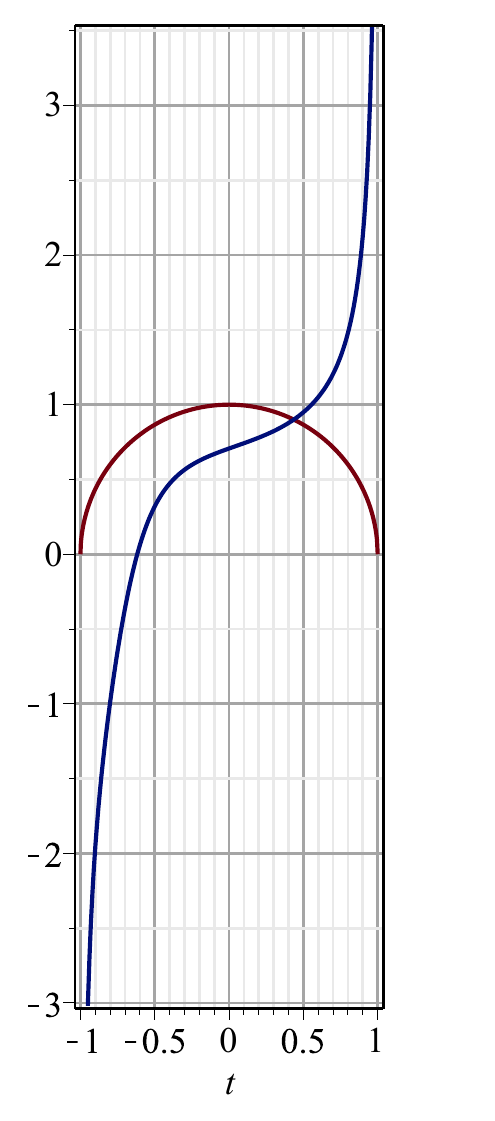}\ \includegraphics[height=6cm,width=4.5cm]{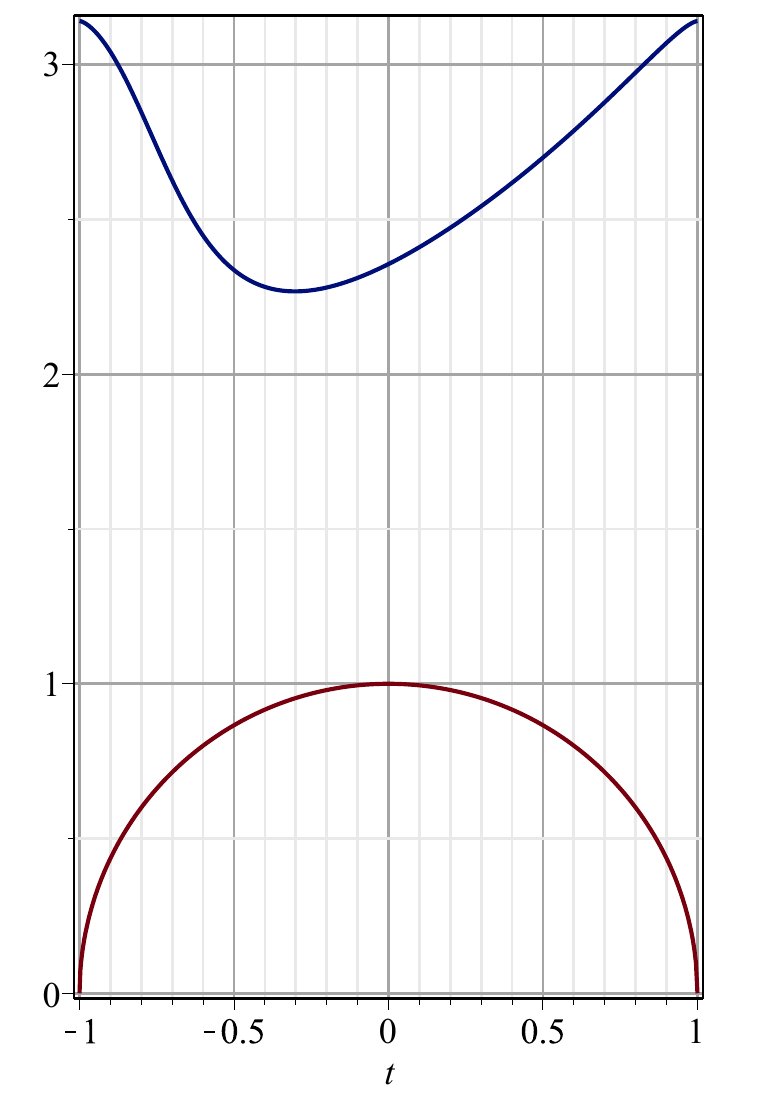}

\begin{center}
{\small \bf Fig. 3. Graphs of $\bsy{x}\ \&\ \bsy{z}$,\ \ $\bsy{x}\ \&\ \bsy{u_1}$,\ \ $\bsy{x}\ \&\ \bsy{\theta}$, respectively, for the upper semi-circle $\ds \bsy{x}=\sqrt{\bsy{1-t^2}}$}
\end{center}

Variables $x,\ z,\ \th$ and $u_2$ remain bounded for $t\in [-1,1]$  but
$$
u_1(t)=\frac{t^2-t-1}{\sqrt{2-2t^2+t^4+2t-2t^3}\sqrt{1-t^2}}
$$
becomes infinite in the limit $t\to\pm 1$, indicating that we are required to change the parametrisation in the neighborhood of points $t=\pm 1$.

\subsection{Flatness of the under-actuated ship}
If a control system is either static feedback linearizable or else some prolongation of it is static feedback linearizable then it satisfies the definition of flatness \cite{LevineBook}. The under-actuated ship is a fully nonlinear control system in four states and two controls. According to the Sluis-Tilbury bound \cite{SluisTilbury}, if (\ref{shipPfaff}) were static feedback linearizable after prolongation then  a maximal 7-fold prolongation in either the $\P {u_1}$ or $\P {u_2}$ directions would be required.  

\begin{prop}
The under-actuated ship is flat, requiring four prolongations in direction $\P {u_2}$.
\end{prop}
\vskip 4 pt
\noindent {\it Proof.} If we differentiate $u_2$ four times in the control system for the underactuated ship $\CE$, (\ref{shipPfaff}), we obtain 
$$
\begin{aligned}
\text{pr}^4\CE=\Big\{\P t+(u_1\cos\theta -z\sin\theta)\P x+(u_1\sin\theta+z\cos\theta)\P y-(&u_1u_2+z)\P z+u_2\P \theta+v_1\P {u_2}+v_2\P {v_1}+\cr
 &v_3\P {v_2}+v_4\P {v_3},\ \P {v_4},\ \P {u_1}\Big\}=\{\mathbf{T},\ \P {v_4},\ \P {u_1}\}.
\end{aligned}
$$
Its refined derived type is
$$
\mathfrak{d}_r(\text{pr}^4\CE)=[\,[3, 0],\ [5, 2, 2],\ [7, 4, 4],\ [9, 6, 7],\ [10, 8, 8],\ [11, 11]\,].
$$
Since $\textrm{decel}(\text{pr}^4\CE)=\langle 0,0,1,0,1\rangle$ (Definition \ref{decel}), this is the refined derived type of $\mathcal{C}\langle 0,0,1,0,1\rangle$. From this we see that among quotients 
$\Xi^{(j)}_{j-1}/\Xi^{(j)}$  only $\Xi^{(3)}_{2}/\Xi^{(3)}$ is non-trivial and 
$\ch\left(\text{pr}^4\CE^{(3)}_2\right)$ is integrable. Since the derived length is $k=5$ with $\D_5=1$, $\text{pr}^4\CE\simeq \mathcal{C}\langle 0,0,1,0,1\rangle$ by Theorem \ref{GGNF}.  We complete the proof by checking that equivalences can be chosen to be static feedback transformations. We do this in a separate calculation below using procedure {\tt\bf Contact}, \cite{Vassiliou2006b}.
\hfill\qed


\vskip 10 pt
\noindent{\it Flat outputs and explicit linearization.} Construct the static feedback linearization using procedure {\bf \tt Contact A} described in \cite{Vassiliou2006b}. The only nontrivial quotient is $\Xi^{(3)}_{2}/\Xi^{(3)}$ spanned by
$$
\a=\big(2u_2^2\cos\theta -(v_1+u_2)\sin\theta\big)\,dx+\big(2u_2^2\sin\theta+(v_1+u_2)\cos\theta\big)\,dy+2u_2\,dz,
$$
and we note that
$$
\Xi^{(3)}=\{dv_1,\ du_2,\ d\theta,\ dt\}.
$$
By the integrability of $\Xi^{(3)}_2$, we have
$$
d\a\equiv 0\mod \a,\ dv_1,\ du_2,\ d\theta.
$$
By the Frobenius theorem, there is a function $a=a(x,y,z,\theta,v_1,u_2)$ and integrating factor $\mu$ such that $\mu da=\a\mod \ dv_1,\ du_2,\ d\theta$.  Function $a$ is a fundamental function of order 3 and therefore one of the two flat outputs. We calculate
$$
a=\big(2yu_2^2-x(v_1+u_2)\big)\sin\theta+\big(2xu_2^2+y(v_1+u_2)\big)\cos\theta+2zu_2
$$

For completeness, we calculate that $\text{ann}\,\Pi^5=\{dt,\ d\theta\}$ in which case the fundamental function of highest order $5$ is $b=\theta$. This is the final flat output. Going on to differentiate $a$ and $b$ by $\mathbf{T}$ (three times for $a$ and five times for $b$) constructs the linearisation. This is easy to invert but the formulas are very complicated and unhelpful for trajectory planning.  Thus, while the under-actuated ship is flat, its flatness does not enable trajectory planning of surface trajectories in an obvious way. Indeed, 
let $s(t)$ and $r(t)$ be arbitrary, $s_i=\partial^i_t s, \ r_j=\partial^j_t r$ with $\ell=s_1^2+4s_1^4+2s_1s_3-3s_2^2$. Then, for instance,
$$
\begin{aligned}
&\ell^2 x=\big(4rs_1^5-6rs_1s_2s_3+2s_1^3r_1+rs_1^3+9rs_2^3+8s_1^5r_1+4r_1s_1^2s_3+2rs_1^2s_3-\cr
&6r_1s_1s_2^2-3rs_1s_2^2-12rs_1^4s_2-3rs_1^2s_2\big)\sin\,s+\big(4r_1s_1^4-4rs_1^6+4r_2s_1^4-rs_1^4+r_1s_1^2+r_2s_1^2+\cr
&3rs_2^2-3r_1s_2^2-3r_2s_2^2-4rs_3^2+2rs_2s_3+3rs_2s_4+4r_1s_2s_3+2r_1s_1s_3-2rs_1s_3-2r_1s_1s_4-\cr
&rs_1s_4-8rs_1^3s_2-16r_1s_1^3s_2+27rs_1^2s_2^2-rs_1s_2-2r_1s_1s_2+2r_2s_1s_3-10rs_1^3s_3\big)\cos\,s.
\end{aligned}
$$
The expressions for $y$ and $z$ have similar complexity. 


\begin{rem}

This example demonstrates the fact that the hypothesis that $\ch\CE^{(j)}_{j-1}$ be integrable cannot be dropped from the definition of Goursat bundle. For suppose we differentiate $u_2$ three times instead of four. Then a calculation reveals the refined derived type of the 3-times prolonged ship control system $\text{pr}^3\CV$ is
$$
\mathfrak{d}_r(\text{pr}^3\CE)=\left[\,\left[3, 0\right],\ \left[5, 2, 2\right],\ \left[7, 4, 4\right],\ \left[9, 6, 7\right],\ \left[10, 10\right]\,\right].
$$
Now $\text{decel}\left(\text{pr}^3\CE\right)=\langle 0, 0, 1, 1 \rangle$ and $\ch\left(\text{pr}^3\CE\right)=\{0\}$. Hence $\text{pr}^3\CE$ has the refined derived type of the partial prolongation 
$\mathcal{C}\langle 0, 0, 1, 1 \rangle$. 
The quotient $\Xi^{(3)}_2/\Xi^{(3)}$ is non-trivial however  $\ch\left(\text{pr}^3\CE^{(3)}_2\right)$ is not integrable! Hence even though $\text{pr}^3\CE$ has the derived type of a partial prolongation  it cannot be equivalent to one.
\begin{rem}
State-space symmetries play a central role in foundational paper \cite{GrizzleMarcus} and such symmetries are used in our subsection \ref{non-flat}. However, the foregoing application demonstrates the importance of the control admissible symmetries introduced in Definition \ref{admissibleSymmetries}, of which state space symmetries form a special subclass. As we saw, in the symmetry reduction and reconstruction of 
(\ref{shipcontrol}), the most useful symmetries are 
non state-space.    
\end{rem}
\end{rem}

\subsection{Application to more sophisticated control systems for ship guidance}

A widely used control system for the guidance of surface marine vessels is the system $\o^i=0, \ 1\leq i\leq 6$ where, 
\begin{equation}\label{shipPfaff1}
\begin{aligned}
&\o^1=dx-(u\cos\th-z\sin\th)dt,\ \ \o^2=dy-(u\sin\th+z\cos\th)dt, \ \ \o^3=d\th-vdt,\cr   
&\o^4=dz+(\g_1 uv+\b_1 z)dt,\ \ \o^5=du-\left(\frac{1}{\g_1} zv-\b_2 u+u_1\right)dt, \ \o^6=dv-\left( \g_2 uz-\b_3 v+u_2\right)dt,
\end{aligned}
\end{equation}
and
$$
\g_2=\frac{m_{11}-m_{22}}{m_{33}},\ \g_1=\frac{m_{22}}{m_{11}},\ \b_1=\frac{d_1}{m_{11}},\ \b_2=\frac{d_2}{m_{22}},\ \b_3=\frac{d_3}{m_{33}}.
$$
Here the $m_{ii}>0$ are components of the mass tensor while the $\b_i$ are the coefficients of hydrodynamic damping. See,  for instance, \cite{Laghrouche}, \cite{fossen}.
In (\ref{shipPfaff1}) just as in (\ref{shipPfaff}), the heave, pitch and roll motions are ignored. However (\ref{shipPfaff1}) extends (\ref{shipPfaff}) by the 1-forms $\{\o^5, \ \o^6\}$, providing a more realistic control system which includes the hydrodynamic damping associated to surge ($u$) and yaw ($v$), in addition to that of sway $(z)$. The only other assumptions made in (\ref{shipPfaff1}) are that the inertia, added mass and hydrodynamic damping matrices \cite{fossen} are diagonal. The properties of the admissible symmetries here as in the previous case have dependence on the choice of parameters. In case $m_{ii}=d_i=1$ for $1\leq i\leq 3$, then the Lie algebra of admissible symmetries of (\ref{shipPfaff1}) is very similar to that of (\ref{shipPfaff}).  For instance, we find an abelian subalgebra of admissible symmetries is given by $\bsy{\G}=\{e^{-t}(-\P y+\sin\th\P u+\cos\th\P z), \P y\}$; {\it cf} (\ref{shipSyms}). Denoting 
$\CV=\ker\bsy{\o}$, where 
$\bsy{\o}=\{\o^i\}_{i=1}^6$, we calculate that the extended distribution $\widehat{\CV}=\CV\oplus\bsy{\G}$ has refined derived type 
$$
\mathfrak{d}_r(\widehat{\CV})=[\,[5,\ 2],\ [7,\ 4,\ 4],\ [9,\ 9]\,].
$$   
From $\mathfrak{d}_r(\widehat{\CV})$ we have $\text{decel}\,\widehat{\CV}=\langle 0, 2\rangle$. Next check that $\widehat{\CV}$ is a relative Goursat bundle. Since $k=2$, we must compute 
$$
\ch\widehat{\CV}^{(1)}=\bsy{\G}\oplus \{\P {u_1},\  \P {u_2}\}
$$
and find that the resolvent bundle is $R(\CV^{(1)})=\{\P u, \P v\}\oplus\ch\widehat{\CV}^{(1)}$ whose integrability implies that $\widehat{\CV}$ is a relative Goursat bundle with signature $\langle 0, 2\rangle$. 
By Theorem  \ref{quotientCheck}, $\CV/G$ is locally diffeomorphic to $\CC\langle 0, 2\rangle$. Since the action is admissible and $2=\dim G<\dim\mathrm{X}(M)=6$, by Theorem \ref{WhenQuotientIsControl} $\CV/G$ is a control system in 4 states and 2 controls. Finally
we observe that $dt\in \text{ann}\,\ch\widehat{\CV}^{(1)}$ and $\{\P {u_1}, \P {u_2}\}\subseteq \ch\widehat{\CV}^{(1)}$ so the hypotheses of Theorem \ref{StatLinQuotients} are satisfied and the quotient $\CV/{G}$ is static feedback linearizable for these parameter values. That is, $\CV/G$ is locally equivalent to Brunovsky normal form $\CC\langle 0,2\rangle$ via a static feedback transformation.

If instead $m_{11}=2$, with all other parameters $m_{ii}, d_j$ taken to be unity then we find that $\bsy{\o}$ has a 4-dimensional Lie algebra of admissible symmetries spanned by
$$
X_1=-y\P x+x\P y+\P \th,\ \ \ \ \ X_2=x\P x+y\P y+u\P u+z\P v+u_1\P {u_1}-2uz\P {u_2},\ \ \ \ \ X_3=\P x, \ \ \ \ \ X_4=\P y.
$$
Now $\bsy{\G}=\{X_2, X_3\}$, for instance, is a (non-abelian) subalgebra. Again, in this case, we find that $\widehat{\CV}=\CV\oplus\bsy{\G}$ is a relative Goursat bundle of signature $\langle 0,2\rangle$ which satisfies Theorem \ref{StatLinQuotients} for the existence of a static feedback linearizable quotient, $\CV/G$. We will report in detail elsewhere on this as well as the further investigation of quotients of under-actuated ship control systems and the planning of surface trajectories using Lie symmetries.

\section{Appendix}

\begin{prop}
Let $(M,\bsy{\o})$ be a Pfaffian system invariant under the regular action on $M$ of Lie group $G$ with quotient map $\mathbf{q}:M\to M/G$ and cross-section $\s:M/G\to M$. Assume that the quotient of $\bsy{\o}$ by $G$ is also a Pfaffian system $\bsy{\bar{\o}}$. Then $\bsy{\bar{\o}}=\s^*\bsy{\o}_{\mathrm{sb}}$, where $\bsy{\o}_{\mathrm{sb}}$ are the $G$-semi-basic 1-forms on $M$.
\end{prop}
\noindent{\it Proof.} If $\bar{\th}\in\bsy{\bar{\o}}$ then $(\mathbf{q}^*\bar{\th})(X)=\bar{\th}(\mathbf{q}_*X)=\bar{\th}(0)=0$ for all $X\in\bsy{\G}$, where $\bsy{\G}$ is the Lie algebra of infinitesimal generators of the $G$-action. Hence there is a semi-basic $\th_{\text{sb}}\in\bsy{\o}_{\text{sb}}$ such that $\mathbf{q}^*\bar{\th}=\th_{\text{sb}}$ and hence
$\bar{\th}=(\mathbf{q}\circ\s)^*\bar{\th}=\s^*\mathbf{q}^*\bar{\th}=\s^*\th_{\text{sb}}$ so that $\bsy{\bar{\o}}\subseteq \s^*\bsy{\o}_{\text{sb}}$.

Fix $\th_{\text{sb}}\in\bsy{\o}_{\text{sb}}$. Then $\th_{\text{sb}}(\bsy{\G})=d\th_{\text{sb}}(\bsy{\G})=0$ and hence $\th_{\text{sb}}$ is the pullback by $\mathbf{q}$ of a 1-form $\bar{\th}$ on $M/G$. Thus
$\th_{\text{sb}}=\mathbf{q}^*\bar{\th}$ implies that $\s^*\th_{\text{sb}}=\s^*\mathbf{q}^*\bar{\th}=(\mathbf{q}\circ\s)^*\bar{\th}=\bar{\th}\in\bsy{\bar{\o}}$. Hence, 
$\s^*\bsy{\o}_{\text{sb}}\subseteq\bsy{\bar{\o}}$.\hfill\qed

\vskip 30 pt
\begin{center}
{\bf Procedure Contact A}\end{center} 
\begin{center}(Case $\D_k=1$)\end{center}

\begin{enumerate}
\item[]{\bf\tt INPUT:} Goursat bundle $\CV\subset TM$ $\big(\bsy{\o}\subset T^*M\big)$ of derived length $k$ and signature
$\sigma=\break\text{\rm decel}(\CV)$=$\langle\r_1,\ldots,\r_k\rangle$,\  $\r_k=1$.

\item[a)] Fix any invariant of $\ch\CV^{(k-1)}$ denoted $x$, and any section $\boldsymbol{Z}$
of $\CV$ such that $\boldsymbol{Z}x=1$.
\item[b)] Build distribution $\Pi^k$, defined inductively by 
$$
\Pi^{l+1}=[Z, \Pi^l],\ \ \ \Pi^1=\ch\CV^{(1)}_0,\ \ \ \ 1\leq l\leq k-1.
$$
\item[c)] Let $z^k=\varphi^{1,k}$ be an invariant of $\Pi^k$ such that $dx\wedge d\varphi^{1,k}\neq 0$. 
\item[d)] For each $j$, such that $\r_j>0$, compute the fundamental bundle 
$\Xi^{(j)}_{j-1}\big/\Xi^{(j)}$ of order $j$.
\item[e)] For each $j$, such that $\Xi^{(j)}_{j-1}\big/\Xi^{(j)}$ is non-trivial, compute the fundamental functions
$\{\varphi^{l_j,j}\}_{l_j=1}^{\r_j}$ of order $j$. [{\it This and step $\mathrm{c)}$ are the only ones requiring integration. The remaining steps require differentiation and linear algebra, alone}.]
\item[f)] For each $j$, such that $\r_j>0$ let $z^{l_j,j}=\varphi^{l_j,j}$, $1\leq l_j\leq \r_j$.
\item[g)] For each $j$, such that $\r_j>0$ define functions
$$
x, z^{l_j,j}_0:=z^{l_j,j}=\varphi^{l_j,j},\ z^{l_j,j}_{s_j+1}=\boldsymbol{Z}z^{l_j,j}_{s_j},\ 0\leq s_j\leq j-1,\ 1\leq l_j\leq \r_j.
$$
\item[]{\bf\tt OUTPUT:} Contact coordinates for $\CV$ \big($\bsy{\o}$\big) identifying it with $\mathcal{C}(\s)$ \big($\mathrm{ann}\,{\CC}(\s)$\big).
\end{enumerate}
\vskip 10 pt
{\it Case $\D_k>1$.} If the Goursat bundle $\CV$ satisfies $\D_k>1$ then steps a) and b) are replaced by the calculation of the resolvent bundle ${R}(\CV^{(k-1)})$ which is integrable and has $1+\D_k$ invariants; these are fundamental functions of highest order, $k$. Any one of these can be taken to be the ``independent" variable, $x$, in the canonical form. We then fix any section $\boldsymbol{Z}$ of $\CV$ such that $\boldsymbol{Z}x=1$ after which we proceed, as in the case $\D_k=1$, to construct contact coordinates.  Proofs of correctness of these procedures are given in \cite{Vassiliou2006b}. For convenience, the stepwise method is described below.

\vskip 20 pt

\begin{center}{\bf Procedure Contact B}\end{center} 
\begin{center}(Case $\D_k>1$)\end{center}
\begin{enumerate}
\item[] {\tt \bf INPUT:} Goursat bundle $\CV\subset TM$ \big(or $\bsy{\o}\subset T^*M$\big) of derived length $k$ and signature
$\s=\text{\rm decel}(\CV)=\langle \r_1,\ldots,\r_k\rangle$,\ $\r_k>1$.
\item[a)] For each $j$, $1\leq j\leq k-1$ such that $\r_j>0$, compute a basis for the fundamental bundle 
$\Xi^{j}_{j-1}\big/\Xi^{j}$ of order $j$.
\item[b)] Compute a basis for fundamental bundle of order $k$:
$\text{ann}\left({R}(\CV^{(k-1)})\right)$
\item[c)] For each $j$, $1\leq j\leq k-1$, such that $\Xi^{j}_{j-1}\big/\Xi^{j}$ non-trivial, compute the fundamental functions $\{\varphi^{l_j,j}\}_{l_j=1}^{\r_j}$ of order $j$ and the fundamental functions of order $k$ from $\text{ann}\left({R}(\CV^{(k-1)})\right)$
\item[d)] Fix any fundamental function of order $k$, denoted $x$ and any section $\mathbf{Z}$ of $\CV$\ \big($\mathbf{Z}$ of $\ker\bsy{\o}$\big) such that $\mathbf{Z}x=1$. [{\it This and step $\mathrm{c)}$ are the only ones requiring integration. The remaining steps require differentiation and linear algebra, alone}.]
\item[e)] For each $j$, such that $\r_j>0$ let $z^{l_j,j}=\varphi^{l_j,j}$, $1\leq l_j\leq \r_j$.
\item[f)] For each $j$, such that $\r_j>0$ define functions
$$
x, \ z^{l_j,j}_0:=z^{l_j,j}=\varphi^{l_j,j},\ z^{l_j,j}_{s_j+1}=\mathbf{Z}z^{l_j,j}_{s_j},\ 0\leq s_j\leq j-1,\ 1\leq l_j\leq\r_j.
$$
\item[] {\tt\bf  OUTPUT:} Contact coordinates for $\CV$ \big($\bsy{\o}$\big) identifying it with $\mathcal{C}(\s)$\ 
\big($\mathrm{ann}\,{\CC}(\s)$\big).
\end{enumerate}

\end{document}